\documentclass[11pt,reqno,a4paper,twoside]{amsart}
\usepackage{fullpage}
\usepackage{fullpage}
\usepackage{tikz}
\usetikzlibrary{fit}
\usepackage{bbm}
\usepackage{amsmath,amsthm}
\usepackage{color}
\usepackage[colorlinks=true,citecolor=black,linkcolor=black,urlcolor=blue]{hyperref} %

\theoremstyle{plain}  %
\numberwithin{equation}{section}
\newtheorem{theorem}{Theorem}[section]  %
\newtheorem{lemma}[theorem]{Lemma}  %
\newtheorem{proposition}[theorem]{Proposition}  %
\theoremstyle{definition}  %
\newtheorem{definition}[theorem]{Definition}  %
\newtheorem{example}[theorem]{Example}  %
\newtheorem{conjecture}[theorem]{Conjecture}  %
\theoremstyle{remark}  %
\newcommand\floor[2]{\left\lfloor \tfrac{#1}{#2} \right\rfloor}
\newcommand\ceil[2]{\left\lceil\tfrac{#1}{#2} \right\rceil}

\newcommand{\Perm}[1]{\mathsf{Perm}\left( #1 \right)}
\renewcommand{\S}{\mathcal{S}}
\newcommand{\T}{\mathcal{T}}

\newcommand{\odd}{\mathrm{odd}}
\newcommand{\even}{\mathrm{even}}

\newcommand\mycom[2]{\genfrac{}{}{0pt}{}{#1}{#2}}
\newcommand{\ds}{\displaystyle}
\newcommand{\des}{\mathrm{des}}
\newcommand{\sn}{\mathfrak{S}_n}
\newcommand{\s}{\mathfrak{S}}
\newcommand{\sym}{\mathfrak{S}}
\newcommand{\GETOUT}[1]{}

\newcommand{\aleft}[2]{\mathsf{left}_{#1}(#2)}

\newcommand{\aright}[2]{\mathsf{right}_{#1}(#2)}
\newcommand{\Tack}[2]{\mathsf{Revstack}_{#1 , #2}}
\newcommand{\Stack}[2]{\mathsf{Stack}_{#1 , #2}}
\newcommand{\rev}{\mathsf{rev}}
\newcommand{\id}{\mbox{id}}
\def\revstack{revstack}

\DeclareMathOperator{\LittlePerm}{{\mathsf{Perms}}}

\def\diagramOne{
        \begin{tikzpicture}[level distance=10mm]
        \tikzstyle{every node}=[fill=white!60,inner sep=1pt]
        \tikzstyle{level 1}=[sibling distance=30mm]
        \tikzstyle{level 2}=[sibling distance=20mm]
        \tikzstyle{level 3}=[sibling distance=10mm]
        \tikzstyle{level 4}=[sibling distance=5mm]
	\node at (-5,-2) {$\Tree(\pi)~=$};
        \node at (0,0) {11}
                child {node {10}
                        child {node {9}
                                child {node {8}
					child {edge from parent[draw=none]}
					child{ node {7}}
					}
                                child{ node {6}
					child{ node {4}}
					child{ node {1}}
					}
                                }
                        child {node {5}
                                child{ node {3}
					child{ node {2}}
					child {edge from parent[draw=none]}
					}
                                child {edge from parent[draw=none]}
                                }
                }
		child {edge from parent[draw=none]};
        \end{tikzpicture}
        }
\def\diagramTwo{
        \begin{tikzpicture}[level distance=10mm]
        \tikzstyle{every node}=[fill=white!60,inner sep=1pt]
        \tikzstyle{level 1}=[sibling distance=30mm]
        \tikzstyle{level 2}=[sibling distance=20mm]
        \tikzstyle{level 3}=[sibling distance=10mm]
        \tikzstyle{level 4}=[sibling distance=5mm]
        \node {$v_{11}$}
                child {node {$v_{10}$}
                        child {node {$v_8$}
                                child {node {$v_5$}
					child {edge from parent[draw=none]}
					child{ node {$v_1$}}
					}
                                child{ node {$v_6$}
					child{ node {$v_2$}}
					child{ node {$v_3$}}
					}
                                }
                        child {node {$v_9$}
                                child{ node {$v_7$}
					child{ node {$v_4$}}
					child {edge from parent[draw=none]}
					}
                                child {edge from parent[draw=none]}
                                }
                }
		child {edge from parent[draw=none]};
        \end{tikzpicture}
        }
\def\diagramThree{
        \begin{tikzpicture}[level distance=10mm]
        \tikzstyle{every node}=[fill=white!60,inner sep=1pt]
        \tikzstyle{level 1}=[sibling distance=10mm]
        \tikzstyle{level 2}=[sibling distance=20mm]
        \tikzstyle{level 3}=[sibling distance=10mm]
        \tikzstyle{level 4}=[sibling distance=5mm]
	\node at (-1,-0.5) {$T_7 ~ =$};
        \node at (0,0) {8}
		child {edge from parent[draw=none]}
		child{ node {7}};
        \node at (2,0) {6}
		child{ node {4}}
		child{ node {1}};
        \node at (4,0) {3}
		child{ node {2}}
		child {edge from parent[draw=none]};
        \end{tikzpicture}
        }
\def\diagramFour{
        \begin{tikzpicture}[level distance=10mm]
        \tikzstyle{every node}=[fill=white!60,inner sep=1pt]
        \tikzstyle{level 1}=[sibling distance=30mm]
        \tikzstyle{level 2}=[sibling distance=20mm]
        \tikzstyle{level 3}=[sibling distance=10mm]
        \tikzstyle{level 4}=[sibling distance=5mm]
	\node at (-5,-2) {$h'(\Tree(\pi))~=$};
        \node at (0,0) {11}
                child {node {10}
                        child {node {9}
                                child {node {8}
					child{ node {7}}
					child {edge from parent[draw=none]}
					}
                                child{ node {6}
					child{ node {4}}
					child{ node {1}}
					}
                                }
                        child {node {5}
                                child {edge from parent[draw=none]}
                                child{ node {3}
					child {edge from parent[draw=none]}
					child{ node {2}}
					}
                                }
                }
		child {edge from parent[draw=none]};
        \end{tikzpicture}
        }

\title{Revstack sort, zigzag patterns, descent polynomials of $t$-revstack sortable permutations, and Steingr\'imsson's sorting conjecture}
\author{Mark Dukes}
\address{Department of Computer and Information Sciences, University of Strathclyde, Glasgow  G1 1XH, UK}
\email{mark.dukes@ccc.oxon.org/mark.dukes@strath.ac.uk}

\begin{document}
\begin{abstract}
In this paper we examine the sorting operator $\T(LnR)=\T(R)\T(L)n$.
Applying this operator to a permutation is equivalent to passing the permutation reversed through a stack.
We prove theorems that characterise $t$-revstack sortability in terms of patterns in a permutation that we call {\it{zigzag}} patterns.
Using these theorems we characterise those permutations of length $n$ 
which are sorted by $t$ applications of $\T$ for $t=0,1,2,n-3,n-2,n-1$.
We derive expressions for the descent polynomials of these six classes of permutations and use
this information to prove Steingr\'imsson's sorting conjecture for those six values of $t$.
Symmetry and unimodality of the descent polynomials for general $t$-revstack sortable permutations is also proven and three conjectures are given.
\end{abstract}
\maketitle

\section{Introduction}
Let $\sn$ be the group of all permutations of $\{1,\ldots ,n\}$.
The classical stack sort operator $\S$
on permutations~\cite{knuthvol1} 
may be defined in two equivalent ways.
One way is to define it as an operator on words representing permutations:
$\S(\epsilon)=\epsilon$ where $\epsilon$ is the empty permutation, $\S(x)=x$ where $x$ is a word of length 1, and
$$\S(LnR) ~=~ \S(L)\S(R)n$$
where $n$ is the unique largest entry in the word $LnR$.
For example, $$\S(42513) = \S(42)\S(13)5 = \S(2)4 \S(1)35 = 24135.$$
Another way to define $\S$ is as the output of pushing $\pi=\pi_1\cdots \pi_n\in\sn$ (from left to right) through a stack.
The stack has the rule that an element may be above another element in the stack only if it is less than it.
If the element at the top of the stack is not less than the next (leftmost) element of $\pi$ which is to next enter the stack,
then elements of the stack must be popped from the stack (to the left) until this is the case.
This is illustrated in Figure~\ref{fig1} for the permutation $\pi=42513$.

\def\mlen{2}
\def\mwidth{0.65}
\def\sscale{1}
\def\mxskip{5.4}
\newcommand{\stackdiagram}[6]{
	\draw (#5 + 0,#6 + 0)--(#5+ \mlen,#6+ 0) -- (#5+\mlen,#6-\sscale*\mlen) -- (#5+\mlen+\mwidth,#6-\sscale*\mlen) -- (#5+\mlen+\mwidth,#6+0);
	\draw (#5+\mlen+\mwidth,#6+0) -- (#5+2*\mlen+\mwidth,#6+0);
	\node at (#5+\mlen+\mwidth,#6+0) (A) {};
	\node at (#5+\mlen,#6+0) (B) {};
	\node at (#5+\mlen+0.5*\mwidth,#6-0.8*\sscale*\mlen) (C) {#3};
	\node at (#5+\mlen+0.5*\mwidth,#6-0.4*\sscale*\mlen) (D) {#2};
	\node[anchor=south west,inner sep=8pt] at (A.south west) {#4};
	\node[anchor=south east ,inner sep=8pt] at (B.south east) {#1};
	}
\def\epsi{0.2}
\def\aheigh{-0.8}
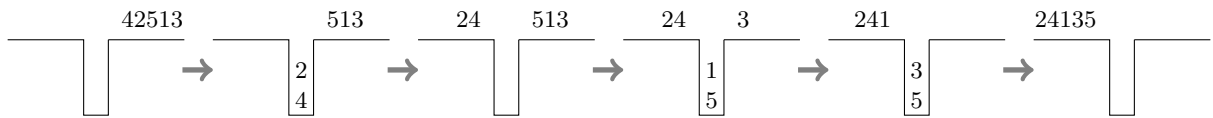
\begin{figure}[!h] 
\begin{tikzpicture}[scale=0.5]
\footnotesize
\stackdiagram{}{}{}{42513}{0}{0}
\stackdiagram{}{2}{4}{513}{1*\mxskip}{0}
\stackdiagram{24}{}{}{513}{2*\mxskip}{0}
\stackdiagram{24}{1}{5}{3}{3*\mxskip}{0}
\stackdiagram{241}{3}{5}{}{4*\mxskip}{0}
\stackdiagram{24135}{}{}{}{5*\mxskip}{0}
\foreach \p in {1,2,3,4,5}
	\draw [gray,->,line width=2pt] (\p*\mxskip - 4*\epsi,\aheigh) -- (\p*\mxskip ,\aheigh);
\end{tikzpicture}
\caption{Moving the permutation $\pi = 42513$ through the stack once to get $\S(\pi)=24135$. 
First, 4 is pushed onto the stack, followed by 2. Since 5 is larger than both 2 and 4, both are popped from the stack before 5 can be pushed onto it.
1 is then pushed onto the stack. As 3 is not less than 1, but is less than 5, 1 is popped from the stack, 3 is pushed onto the stack, and then all stack elements are popped from the stack.
\label{fig1}}
\end{figure}
It is clear that repeatedly applying $\S$ to a permutation will yield the identity permutation, and the operator $\S$ is therefore a sorting operator.
The set of permutations which are sorted by at most one application of $\S$ is known to be the set of 231-avoiding permutations, 
and this result is often described as the first result in the combinatorics of permutation patterns.
We will assume the reader is familiar with the basic terminology of permutation patterns~\cite{bona.book}.

Variants of this operation and their analysis have also thrown up interesting connections to other areas of discrete mathematics. 
Recently, the author in collaboration with others showed how the sorting operation $B(LnR)=B(L)Rn$ describes one pass of the bubble-sort operator. 
Further to this, it was shown that permutations in $\sn$ which require at most $t$ passes of $B$ are in bijection with the 
set of ground state juggling sequences of period $n$ with $t$ balls~\cite{chung}.

In this paper, we will consider another variant of the stack sort operator. We will call this variant {\it{\revstack\ sort}} and 
denote it by $\T = \S \circ \rev$, where $\rev$ is the operation that reverses the list that it acts upon.
We may write a recursion for $\T$ as we did for $\S$ above:
Given $\pi = L n R \in \sn$, $\T$ satisfies 
$$\T(LnR) ~=~ \T(R)\T(L)n,$$
$\T(\epsilon)=\epsilon$, and $\T(x)=x$ where $x$ is a word of length 1.
For example, $\T(42513)=\T(13)\T(42)5 = \T(1)3 \T(2) 4 5 = 13245$.
An equivalent way to describe this operation is that it is what results when a permutation is fed backwards into a stack.
For $\pi = 42513$, this is illustrated in Figure~\ref{fig2}.

This new sorting operation is interesting because it appears to be faster than stack sort in the following sense:
\begin{conjecture}[Steingr\'imsson's sorting conjecture~\cite{einar}] \label{econj}
For all $0 \leq t \leq n$,
$$|\{ \pi \in \sn ~:~ \S^t(\pi) = \mathrm{id} \}| ~\leq ~ \{ \pi \in \sn ~:~ \T^t(\pi) = \mathrm{id} \}|,$$
where $\mathrm{id}$ is the identity permutation. Furthermore, this inequality is strict for all pairs $(n,t)$ that satisfy $2<t<n-1$.
\end{conjecture}

In Section 2 we will define a $t$-zigzag of a permutation. We present theorems which tell us when permutations are and are not $t$-revstack sortable based on whether or not the permutation contains (certain types of) zigzags of certain degrees.
In Section 3 we use the theorems of Section 2 to classify permutations of length $n$ that are 1-revstack, 2-revstack, $(n-3)$-revstack and $(n-2)$-revstack sortable.
These classifications are then used to derive the descent polynomials which are the generating functions of the descent statistic on classes of $t$-\revstack\ sortable permutations for $t=0,1,2,n-3,n-2$ and $n-1$.

These generating functions are then used in Section 4 to prove Steingr\'imsson's conjecture those six particular values of $t$.
In Section 5 we prove symmetry and unimodality of the descent polynomials for the class of $t$-\revstack\ sortable permutations,
and conjecture log-concavity and real-rooted-ness of coefficients of our polynomials based on the evidence given in the Appendix.

\begin{figure}
\begin{tikzpicture}[scale=0.5]
\footnotesize
\stackdiagram{}{}{}{42513}{0}{0}
\stackdiagram{}{1}{3}{425}{1*\mxskip}{0}
\stackdiagram{13}{}{5}{42}{2*\mxskip}{0}
\stackdiagram{13}{2}{5}{4}{3*\mxskip}{0}
\stackdiagram{132}{4}{5}{}{4*\mxskip}{0}
\stackdiagram{13245}{}{}{}{5*\mxskip}{0}
\foreach \p in {1,2,3,4,5}
	\draw [gray,->,line width=2pt] (\p*\mxskip - 4*\epsi,\aheigh) -- (\p*\mxskip ,\aheigh);
\end{tikzpicture}
\caption{Reversing the permutation $\pi= 42513$ and then moving it through a stack is equivalent to passing the rightmost elements of the permutation to the stack.
From this we get $\T(42513)=13245$.
\label{fig2}}
\end{figure}
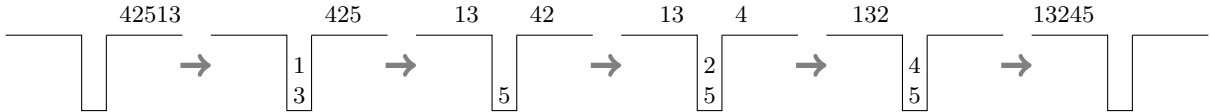

Let $\Tack{n}{t} = \{ \pi \in \sn ~:~ \T^t(\pi)=\mbox{\id}\}$, the set of $t$-\revstack\ sortable permutations.
Given $\pi \in \sn$, let $\deg_{\T}(\pi) = \inf\{ i ~:~ \T^i(\pi)=\id \mbox{ and } \T^{i-1}(\pi) \neq \id\}$, the number
of applications of $\T$ that are required to sort $\pi$.

\section{Classification theorems and zigzag patterns}
In this section we will present classification theorems which assist us in deciding whether or not a permutation is $t$-revstack sortable. 
These classifications are in terms of {\it{zigzags}} which are subsequences of values in the permutation.
These will be used to classify particular classes of $t$-\revstack\ sortable permutations.

\begin{lemma}\label{LEMMA:1}  
If $\pi \in \sn$, $1\leq a < b \leq n$, and $b$ precedes $a$ in $\pi$, 
then $a$ precedes $b$ in $\T(\pi)$ 
\end{lemma}

\begin{proof}
Since $b$ precedes $a$ in $\pi$, $a$ enters the stack before $b$. Since $a<b$, element $a$ must have been popped from the stack before element $b$ enters. 
Thus $a$ precedes $b$ in $\T(\pi)$.
\end{proof}

\begin{lemma}\label{LEMMA:2}  
Suppose that $\pi \in \sn$, $1\leq a<b \leq n$, and $a$ precedes $b$ in $\pi$.
Then $b$ precedes $a$ in $\T(\pi)$ if there exists $c>b$ such that 
$a$ precedes $c$ and $c$ precedes $b$ in $\pi$.
If no such $c$ exists then $a$ precedes $b$ in $\T(\pi)$.
\end{lemma}

\begin{proof}
If there is an element $c>b$ such that 
$a$ precedes $c$ and $c$ precedes $b$ in $\pi$, 
then $b$ will be the first of the three elements to enter 
the stack. Since $c>b$, element $b$ must be popped from the stack before 
$c$ enters the stack. It is only after $c$ enters the stack that $a$ may enter the stack and be popped. 
Thus $b$ precedes $a$ in $\T(\pi)$.

If no such $c$ exists then all elements that appear between $a$ and $b$ in $\pi$ are less than $b$. 
Element $b$ will enter the stack and remain in the stack while $a$ enters. 
Element $a$ must be popped before element $b$, therefore $a$ precedes $b$ in $\T(\pi)$.
\end{proof}

\begin{lemma}\label{LEMMA:3} 
$(b,a)$ is an inversion in $\T(\pi)$ iff there exists $c$ such that 
$(a,c,b)$ is a $132$ pattern in $\pi$.
\end{lemma}

\begin{proof}
Lemma~\ref{LEMMA:1} tells us that an inversion $(b,a)$ in $\pi$ becomes a non-inversion $(a,b)$ in $\T(\pi)$.
Lemma~\ref{LEMMA:2} tells us that a non-inversion $(a,b)$ in $\pi$ becomes an inversion $(b,a)$ in $\T(\pi)$ 
iff an element $c$ larger than both exists between the two in $\pi$.
Thus if $(b,a)$ is an inversion in $\T(\pi)$ then there exists $c$ such that
$(a,c,b)$ is an occurrence of the pattern $132$ in $\pi$.
If $(a,c,b)$ is a 132 pattern in $\pi$, then applying $\T$ one finds that $b$ precedes $a$ in $\T(\pi)$, thereby forming an inversion $(b,a)$ in $\T(\pi)$.
\end{proof}

\begin{theorem}\label{THEOREM:1}
$\Tack{n}{1} = \sn(132)$.
\end{theorem}

\begin{proof}
A permutation $\pi \in \sn$ is in $\Tack{n}{1}$ iff $\T(\pi)=\id$, which occurs iff $\T(\pi)$ has no inversions.  
Lemma~\ref{LEMMA:3} tells us this happens precisely when $\pi$ is 132 avoiding.
\end{proof}

\begin{theorem}\label{mrbluesky}
(i) If $\T(\sigma)$ contains $(a,c,b)$ as a $132$ pattern, then either $(b,d,c,a)$ is a $2431$ pattern in $\sigma$, 
or $(b,d,a,c)$ is a $2413$ pattern in $\sigma$ and there is no $e>c$ such that $(b,d,a,e,c)$ in $\sigma$.\\
(ii) 
If $\T(\sigma)$ contains a $132$ pattern, then $\sigma$ either contains a $2431$ pattern, or a $241\overline{5}3$ pattern.
\end{theorem}

\begin{proof}
Lemmas~\ref{LEMMA:1}--\ref{LEMMA:3} provide us with the following information about patterns in $\sigma$ and $\T(\sigma)$.
We will say `$(x,y)$ in $\sigma$' to mean $x$ precedes $y$ in the the word representing the permutation $\sigma$,
and `$(x,y,z)$ in $\sigma$' to mean $x$ precedes $y$ and $y$ precedes $z$ in the word representing the permutation $\sigma$, etc
\begin{itemize}
\item If $(b,a)$ in $\sigma$, then $(a,b)$ in $\T(\sigma)$.
\item If $(a,b)$ in $\sigma$, then
	\begin{itemize}
	\item if $\exists$ $c>b$ such that $(a,c,b)$ in $\sigma$ then $(b,a)$ in $\T(\sigma)$,
	\item if $\not\!\exists$ $c>b$ such that $(a,c,b)$ in $\sigma$ then $(a,b)$ in $\T(\sigma)$.
	\end{itemize}
\end{itemize}
Rewriting these observations in terms of how entries appear in $\T(\sigma)$ first:
\begin{itemize}
\item If $(a,b)$ in $\T(\sigma)$ then either,
	\begin{itemize}
	\item $(b,a)$ in $\sigma$, or
	\item $(a,b)$ in $\sigma$ and $\not\!\exists$ $c>b$ such that $(a,c,b)$ in $\sigma$.
	\end{itemize}
\item If $(b,a)$ in $\T(\sigma)$ then $\exists$ $c>b$ such that $(a,c,b)$ in $\sigma$.
\end{itemize}
Using these observations, if $\T(\sigma)$ contains $(a,c,b)$ as a 132 pattern then the following hold:
\begin{enumerate}
\item[1.] Since $(c,b)$ in $\T(\sigma)$, there exists $d>c$ such that $(b,d,c)$ in $\sigma$.
\item[2.] Since $(a,c)$ in $\T(\sigma)$, either 2(i) $(c,a)$ in $\sigma$, or 2(ii) $(a,c)$ in $\sigma$ and $\not\!\exists$ $e>c$ such that $(a,e,c)$ in $\sigma$.
\item[3.] Since $(a,b)$ in $\T(\sigma)$, either 3(i) $(b,a)$ in $\sigma$, or 3(ii) $(a,b)$ in $\sigma$ and $\not\!\exists$ $f>b$ such that $(a,f,b)$ in $\sigma$.
\end{enumerate}
Let us now check all possible combinations of 2(i), 2(ii), 3(i), and 3(ii).
\begin{description}
\item[1, 2(i) and 3(i)] We have $(b,d,c)$ in $\sigma$, $(c,a)$ in $\sigma$, and $(b,a)$ in $\sigma$. These imply that $(b,d,c,a)$ in $\sigma$ where $d>c$.
\item[1, 2(i) and 3(ii)] We have $(b,d,c)$ in $\sigma$, $(c,a)$ in $\sigma$, and $(a,b)$ in $\sigma$ with no $f>b$ such that $(a,f,b)$ in $\sigma$. 
	These are inconsistent as the first two imply that $(b,a)$ in $\sigma$ whereas the third says that $(a,b)$ in $\sigma$. This is impossible,
\item[1, 2(ii) and 3(i)] We have $(b,d,c)$ in $\sigma$, $(b,a)$ in $\sigma$, and $(a,c)$ where $\not\!\exists e>c$ such that $(a,e,c)$ in $\sigma$. 
	This implies that $(b,d,a,c)$ in $\sigma$ and $\not\!\exists e>c$ such that $(b,d,a,e,c)$.
\item[1, 2(ii) and 3(ii)] We have $(b,d,c)$ in $\sigma$, 
$(a,c)$ in $\sigma$ where $\not\!\exists$ $e>c$ such that $(a,e,c)$ in $\sigma$,
and $(a,b)$ in $\sigma$ where $\not\!\exists$ $f>b$ such that $(a,f,b)$ in $\sigma$.
These three are inconsistent and this case is therefore not possible.
\end{description}
We conclude that if $\T(\sigma)$ contains $(a,c,b)$ as a 132 pattern, then either
\begin{itemize}
\item $(b,d,c,a)$ is a 2431 pattern in $\sigma$, or
\item $(b,d,a,c)$ is a 2413 pattern in $\sigma$ and $\not\!\exists$ $e>c$ such that $(b,d,a,e,c)$ in $\sigma$.
\end{itemize}
Notice that in the second condition, if $e$ is such that $c<e<d$, then $(b,d,a,e)$ is a 2413 pattern. Therefore part (i) is true.

At the very end of the previous argument, we are told that there does not exist $e>c$ such that $(b,d,a,e,c)$ in $\sigma$.
Since $d$ is greater than $c$ and less than $e$, this statement implies that there does not exist $e>d$ such that $(b,d,a,e,c)$ in $\sigma$.
This implies part (ii) of the theorem.
\end{proof}

\begin{definition}\label{zigzag:defn}
A $k$-zigzag pattern in a permutation $\pi$ is a sequence $z=(z_0,\ldots,z_{k+1})$ such that
\begin{enumerate}
\item[(a)] $z_0 > z_1>z_2>\cdots > z_{k+1}$, and
\item[(b)] $\pi^{-1}(z_{2i}) < \pi^{-1}(z_0) < \pi^{-1}(z_{2i-1})$ for all $i\geq 1$.
\end{enumerate}
\end{definition}

A 0-zigzag pattern is an occurrence of the permutation pattern 21, and so the only permutation that contains no 0-zigzag pattern is the identity.
A 1-zigzag pattern is an occurrence of the permutation pattern 132.
A 2-zigzag pattern is either an occurrence of the permutation pattern 2413 or an occurrence of the permutation pattern 2431.
A 3-zigzag pattern is illustrated in Figure~\ref{fig:three}.

\begin{figure}
\begin{center}
\includegraphics[scale=0.6]{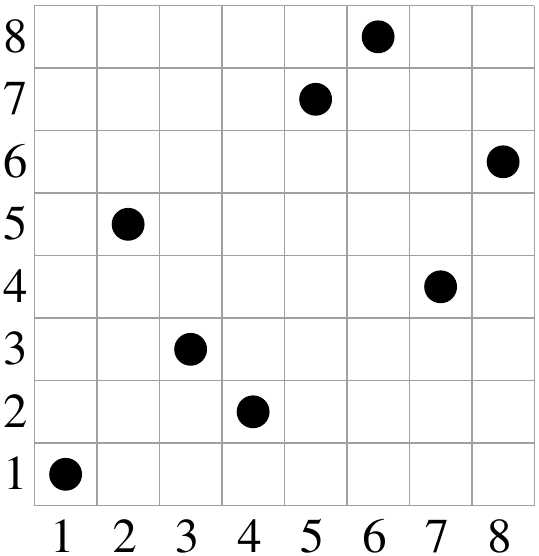}
$\hspace*{2em}$
\includegraphics[scale=0.6]{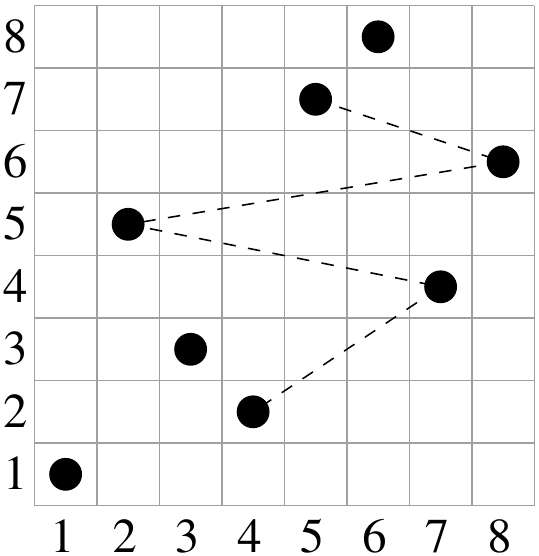}
\end{center}
\caption{The permutation 15327846 contains a 3-zigzag pattern $x=(7,6,5,4,2)$. We have drawn lines between the points in the permutation diagram which correspond to adjacent elements in the 3-zigzag.\label{fig:three}
}
\end{figure}

\begin{theorem}\label{thm:nozigzag}
If $\pi \in \sn$ contains no $k$-zigzag pattern, then $\T^{k}(\pi)=\id$.
\end{theorem}

\begin{proof}
If $\pi$ contains no $1$-zigzag pattern, then $\pi$ avoids the pattern $132$, which by Theorem~\ref{THEOREM:1} means $\T(\pi)=\T^1(\pi)=\id$. 
Therefore the result of the theorem is true for $k=1$.

The proof will follow by induction. 
It will be sufficient to show that given $\pi \in \sn$,\\[1em]
\centerline{$\pi$ contains no $k$-zigzag pattern $\Longrightarrow$ $\T(\pi)$ contains no $(k-1)$-zigzag pattern.}
\ \\
We will prove the contrapositive to the above statement:\\[1em]
\centerline{$\T(\pi)$ contains a $(k-1)$-zigzag pattern $\Longrightarrow$ $\pi$ contains a $k$-zigzag pattern.}
\ \\
Suppose that $z=(z_0,\ldots,z_k)$ is a $(k-1)$-zigzag pattern in $\T(\pi)$.
Since $\{(z_0,z_{2i-1})\}_{i\geq 1}$ are inversions in $\T(\pi)$, 
by Lemma~\ref{LEMMA:3} we know there must exist values $c_i$ such that $(z_{2i-1},c_i,z_0)$ are occurrences of 132 patterns in $\pi$.
Let $c^{*}$ be the rightmost such $c_i$ in $\pi$. 
Then $(z_{2i-1}, c^{*},z_0)$ are also all occurrences of 132 patterns in $\pi$ which means
$\pi^{-1}(z_{2i-1}) < \pi^{-1}(c^{*})$ for all $i$.

Next suppose that there is some $z_{2j}$ (where $j\ \geq 1$) such that
$(z_{2j},c^{*})$ is a 12 pattern in $\pi$.
Then $(z_{2j},c^{*},z_0)$ is a 132 pattern in $\pi$ by Definition~\ref{zigzag:defn} which, according to 
Lemma~\ref{LEMMA:3}, means $(z_0,z_{2j})$ is a 21 pattern in $\T(\pi)$.
This is a contradiction since $z_{2j}$ precedes $z_0$ in $\T(\pi)$. 
We conclude that all even indexed $z_i$'s are to the right of $c^{*}$ in $\pi$, hence $\pi^{-1}(c^{*}) < \pi^{-1}(z_{2j})$ for all $j\geq 1$.

The two pairs of inequalities above, together with $c^{*} > z_0 > z_1 > \cdots > z_{k}$, show that 
$z'=(c^{*},z_0,\ldots ,z_{k})$ is a $k$-zigzag pattern in $\pi$.
\end{proof}

\begin{definition}
Let $z=(z_0,\ldots,z_{k+1})$ be a $k$-zigzag pattern in $\pi \in \sn$.
We will call $z$ {\it{interrupted}} if there exists $c>z_0$ such that either 
\begin{enumerate}
\item[(i)] $z_{2a}$ precedes $c$ and $c$ precedes $z_{2b}$ in $\pi$ for some $a,b \geq 1$, or 
\item[(ii)] $z_{2a-1}$ precedes $c$ and $c$ precedes $z_{2b-1}$ in $\pi$ for some $a,b \geq 1$.
\end{enumerate}
Otherwise we will call $z$ {\it{uninterrupted}}.
\end{definition}

\begin{example}
The $3$-zigzag $z=(7,6,5,4,2)$ in the permutation $\pi=(1,5,3,2,7,8,4,6)$ shown in Figure~\ref{fig:three} is uninterrupted.
The $3$-zigzag $z'=(7,6,5,4,2)$ in the permutation $\pi'=(1,5,3,2,7,4,8,6)$ shown is interrupted because there is an $8$ between $4$ and $6$.
The $4$-zigzag $z''=(10,9,8,7,6,5)$ in the permutation
$\pi'' = (4,6,11,8,3,2,10,12,7,1,9,5)$
is interrupted because there is an $11$ between $8$ and $6$ in $\pi''$.
See also Figure~\ref{unint_example}.
\begin{figure}[h]
\centerline{\includegraphics[scale=0.75]{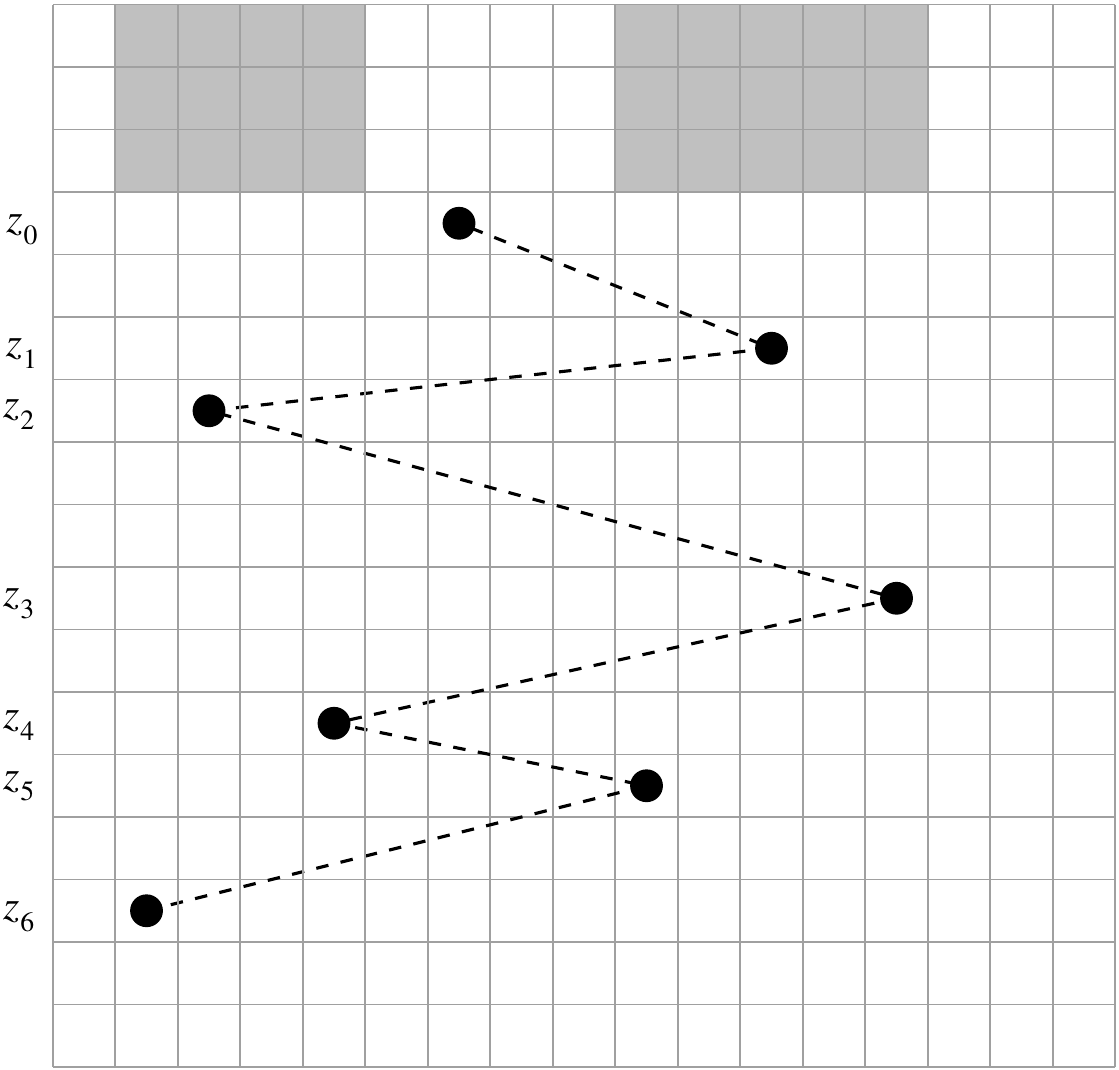}}
\caption{An illustration of where points are forbidden (the grey areas) in order for a $5$-zigzag $z=(z_0,z_1,\ldots,z_6)$ to be an uninterrupted 5-zigzag.
\label{unint_example}
}
\end{figure}
\end{example}

\begin{definition}
Given $\pi \in \sn$ and a set of elements $A \subseteq \{1,\ldots,n\}$, let
$\aleft{\pi}{A}$ (resp. $\aright{\pi}{A}$) be the element of $A$ that is leftmost (resp. rightmost) in $\pi$.
We will say that an element $x$ is {\it{among}} $A$ in $\pi$ if $x$ does not precede $\aleft{\pi}{A}$ in $\pi$ and 
$\aright{\pi}{A}$ does not precede $x$ in $\pi$.
\end{definition}

For example if $\pi=285491376$ and $A=\{1,3,5,9\}$ then $\aleft{\pi}{A}=5$ and $\aright{\pi}{A}=3$.
The elements 4 and 5 are among $A$ in $\pi$ whereas the elements 2 and 7 is not.

\begin{theorem}\label{thm:zigzag}
If $\pi \in \sn$ contains an uninterrupted $k$-zigzag then $\T^k(\pi) \neq \id$.
\end{theorem}

\begin{proof}
An uninterrupted 1-zigzag in a permutation $\pi \in \sn$ is an occurrence of the pattern 132.
If $\pi \in \sn$ contains an occurrence of the pattern 132 then $\T(\pi) \neq \id$ by Theorem~\ref{THEOREM:1}.
The theorem is true for $k=1$.

Now we claim that if $\pi$ contains an uninterrupted $k$-zigzag, then $\T(\pi)$ contains an uninterrupted $(k-1)$-zigzag.
To see the validity of this claim consider the following.
Let $z=(z_0,z_1,\ldots,z_{k+1})$ be the lexicographically largest $k$-zigzag in $\pi$ that begins with $z_0$ that is uninterrupted.
Let $z_{\odd}=\{z_{2j-1}:j\geq 1\}$ and $z_{\even} = \{z_{2j}: j\geq 1\}$.

All elements of $z_{\odd}$ enter and leave the stack before $z_0$ does.
There are no elements larger than $z_0$ among $z_{\odd}$ in $\T(\pi)$ as $z$ is uninterrupted (the elements among $z_{\odd}$ in $\T(\pi)$ start by some element among $z_{\odd}$ in $\pi$ and end by the maximal element among $z_{\odd}$).
Similarly there are no elements larger than $z_0$ among $z_{\even}$ in $\T(\pi)$.

Since $z$ is the lexicographically largest $k$-zigzag with the desired properties, we know that there are no elements larger than $z_1$ among $z_{\even}$.
Also $z_1$ is, by definition, the largest element of $z_{\odd}$ that is less than $z_0$ (otherwise there would exist a lexicographically 
larger $k$-zigzag $z'$ which starts with $z_0$ in $\pi$).
From this we know that $z_1$ is the rightmost element among $z_{\odd}$ in $\T(\pi)$.
The sequence $(z_1,\ldots,z_{k+1})$ is therefore an uninterrupted $(k-1)$-zigzag in $\T(\pi)$.

Iterating the above claim, we have that
if $\pi$ contains an uninterrupted $k$-zigzag, 
then $\T(\pi)$ contains an uninterrupted $(k-1)$-zigzag,
$\T^2(\pi)$ contains an uninterrupted $(k-2)$-zigzag,
and
$\T^{k}(\pi)$ contains an uninterrupted $0$-zigzag.
An uninterrupted $0$-zigzag is simply an inversion, so $\T^k(\pi)$ contains an inversion and is therefore not the identity permutation.
\end{proof}

The following result is the companion to West's result~\cite{west.phd} which tells us that the set of 2-stack sortable permutations is $\sn(2341,3\overline{5}241)$.

\begin{theorem}
$\Tack{n}{2} = \sn(2431,241\overline{5}3)$.
\end{theorem}

\begin{proof}
Suppose that $\pi \in \sn$ with $\T^2(\pi)=\id$.
According to Theorem~\ref{thm:zigzag}, $\pi$ does not contain an uninterrupted 2-zigzag.
There are two 2-zigzags: 2431 and 2413.
The statement `$\pi$ does not contain an uninterrupted 2-zigzag' is equivalent to `$\pi$ does not contain a 2-zigzag, or if it does then it is interrupted', and this happens iff
$\pi \in \sn(243\overline{5}1) \cap \sn (241\overline{5} 3)$.

Notice, however, that if $\pi$ contains $(b,d,c,e,a)$, where $a<b<c<d<e$, as $(d,c,b,a)$ is an interrupted 2-zigzag then in the permutation $\T(\pi)$;
$a$ will precede $c$ (by Lemma~\ref{LEMMA:3}) and $c$ will precede $b$ (by Lemma~\ref{LEMMA:2}).
This means $(a,c,b)$ is a 132 pattern in $\T(\pi)$.
By Theorem~\ref{THEOREM:1}
$\T^2(\pi) \neq \id$.
Therefore $\pi$ cannot contain the pattern 24351 and so
\begin{align}\label{meone}
\Tack{n}{2} \subseteq \sn(243\overline{5}1) \cap \sn(241\overline{5}3) - \sn(24351) = \sn(2431,241\overline{5}3).
\end{align}
To complete the proof we must show that $\T^2(\pi) \neq \id$ $\Rightarrow$ $\pi$ contains 2431, or $\pi$ contains $241\overline{5}3$.
Suppose that $\T^2(\pi) \neq \id$. Using Theorem~\ref{THEOREM:1} this implies that $\T(\pi)$ contains a 132 pattern. 
Using Theorem~\ref{mrbluesky}(ii) we have that $\pi$ contains the pattern $2431$, or the pattern $241\overline{5}3$.
We may summarise this as
\begin{align}\label{metwo}
\sn \backslash \Tack{n}{2} \subseteq \sn\backslash \sn(2431,241\overline{5}3).
\end{align}
Combining Equations~\ref{meone} and \ref{metwo} we have 
$\Tack{n}{2} =\sn(2431,241\overline{5}3)$. 
\end{proof}

\section{Descent polynomials}
In this section we will determine the descent polynomial over $\Tack{n}{k}$ for several values of $k$ relative to $n$.
Given $A \subseteq \sn$ define 
$$W(A;x) ~=~ \sum_{\pi \in A} x^{1+\des(\pi)}.$$

It is straightforward to see that $\Tack{n}{n-1} = \sn$ and $W(\Tack{n}{n-1};x)$ is therefore $A_n(x)$,
the $n$th Eulerian polynomial.

\begin{theorem}\label{onestack}
$W(\Tack{n}{1};x) = \ds\sum_{0\leq k <n}  \frac{1}{n} \binom{n}{k}\binom{n}{k+1} x^{n-k}= W(\Stack{n}{1};x)$.
\end{theorem}

\begin{proof}
Recall that $\Stack{n}{1}=\sn(231)$.
Simion~\cite{simion} showed that $$W(\sn(231);x) = \ds\sum_{k=0}^{n-1} N(n,k)x^{k+1}$$ 
where $N(n,k)=\frac{1}{n}\binom{n}{k}\binom{n}{k+1}$ are the well-known Narayana numbers.
Theorem~\ref{THEOREM:1} showed that $\Tack{n}{1}$ $=$ $\sn(132)$.
Since $\sn(132)$ is easily formed from $\sn(231)$ by reversing the elements, and by using the symmetry of the Narayana numbers $N(n,k)=N(n,n-k-1)$,
we have 
\begin{align*}
W(\sn(132);x) &
= \sum_{\pi \in \sn(231)} x^{1+\des(\rev(\pi))} 
= \sum_{\pi \in \sn(231)} x^{n-\des(\pi)} 
= W(\Stack{n}{1};x).\qedhere
\end{align*}
\end{proof}

\begin{theorem}\label{twostack}
$W(\Tack{n}{2};x) = W(\Stack{n}{2};x)$. 
\end{theorem}

\begin{proof}
By definition,
\begin{align*}
W(\Tack{n}{2};x) 
&= \ds \sum_{\mycom{\pi \in \sn}{S\circ\mathrm{rev} \circ S \circ \mathrm{rev}(\pi)=\mathrm{id}}} x^{1+\des(\pi)} \\
&= \ds \sum_{\mycom{\sigma \in \sn}{S\circ \mathrm{rev} \circ S (\sigma)=\mathrm{id}}} x^{1+\des(\rev(\sigma))} \\
&= \ds \sum_{\mycom{\sigma \in \sn}{S\circ \mathrm{rev} \circ S (\sigma)=\mathrm{id}}} x^{n-\des(\sigma)} \\
&= x^{n}\ds\sum_{\mycom{\sigma \in \sn}{S\circ\mathrm{rev} \circ S (\sigma)=\mathrm{id}}} (x^{-1})^{\des(\sigma)}.
\end{align*}
Bouvel and Guibert~\cite[Thm. 1.3]{bouvel} recently showed that the descent statistic is equi-distributed on the sets
$\{\pi \in \sn ~:~ \S\circ \S(\pi)=\mathrm{id}\}$
and
$\{\pi \in \sn ~:~ \S\circ \mathrm{rev}\circ \S(\pi)=\mathrm{id}\}$.
This means 
\begin{align*}
x^{n}\ds\sum_{\mycom{\sigma \in \sn}{S\circ \mathrm{rev} \circ S (\sigma)=\mathrm{id}}} (x^{-1})^{\des(\sigma)}
&=
x^{n+1}\ds\sum_{\mycom{\sigma \in \sn}{S^2 (\sigma)=\mathrm{id}}} (x^{-1})^{1+\des(\sigma)}\\
&= x^{n+1} W(\Stack{n}{2};x^{-1})\\
&=W(\Stack{n}{2};x),
\end{align*}
the final equality stemming from the fact that $W(\Stack{n}{2};x)$ is symmetric in the sense that $[x^k]W(\Stack{n}{2};x) = [x^{n+1-k}] W(\Stack{n}{2};x)$ for all $1\leq k \leq n$.
This symmetry is easily checked from the expression for $W(\Stack{n}{2};x)$ given in Dulucq et al.~\cite[Cor. 9]{dulucq}.
\end{proof}

\newcommand{\dtwo}[1]{[#1]^{\downarrow 2}}
Throughout the remainder of this section we will employ some new terminology. 
Given a set $X=\{x_1,\ldots,x_m\}$, let $\Perm{X}$ be the set $\{ x_{\sigma(1)}x_{\sigma(2)} \cdots x_{\sigma(m)} ~:~ \sigma \in \s_m\}$.
For example, $\Perm{\{5,4,1\}} = \{145,154,415,451,514,541\}$.
Given a positive integer $m$, let $\dtwo{m}$ be the set of positive integers no larger than $m$ and which differ from it by an even number:
$$\dtwo{m}=\{m,m-2,m-4,\ldots,1+((m-1)\mod 2)\}.$$
Many of our cases will involve $\dtwo{m}$ for certain values $m$.

\begin{theorem} 
\label{rl:nm1}
For all $n\geq 4$, 
\begin{eqnarray*}
W(\Tack{n}{n-2};x) &=& A_n(x) - A_{\left\lceil \frac{n-1}{2}\right\rceil} (x) A_{\left\lfloor \frac{n-1}{2} \right\rfloor}(x).
\end{eqnarray*}
\end{theorem}

\begin{proof}
We wish to obtain an expression for
\begin{align*}
W(\Tack{n}{n-2};x) &= W(\Tack{n}{n-1};x) - \sum_{\mycom{\pi \in \sn}{\deg_{\T}(\pi)=n-1}} x^{1+\des(\pi)}\\
&= A_n(x) - \sum_{\mycom{\pi \in \sn}{\deg_{\T}(\pi)=n-1}} x^{1+\des(\pi)}.
\end{align*}
We will now classify those permutations $\pi \in \sn$ for which $\deg_{\T}(\pi)=n-1$. 
Suppose that $\pi \in \sn$ with $\deg_{\T}(\pi)= n-1$. Then $\T^{n-1}(\pi)=\id$ and $\T^{n-2}(\pi) \neq \id$.
The contrapositive of Theorem~\ref{thm:nozigzag} tells us that $\pi$ must contain an $(n-2)$-zigzag.
An $(n-2)$-zigzag in a permutation $\pi$ is a strictly decreasing sequence $z=(z_0,z_1,\ldots, z_{n-1})$. 
Because $z$ contains $n$ elements out of a possible $n$ elements, it is the sequence $(n,n-1,\ldots,1)$.
Therefore $\pi = LnR$ where $L \in \Perm{\{n-2,n-4,\ldots\}} = \Perm{\dtwo{n-2}}$ and $R \in \Perm{\{n-1,n-3,\ldots\}} = \Perm{\dtwo{n-1}}$.
In this case $L$ has $\lceil (n-1)/2\rceil$ elements and $R$ has $\lfloor (n-1)/2 \rfloor$ elements.
Therefore
\begin{align*}
W(\Tack{n}{n-2};x)
&= A_n(x) - \sum_{\mycom{L\in\dtwo{n-2}}{R\in\dtwo{n-1}}} x^{1+\des(LnR)}\\
&= A_n(x) - \sum_{\mycom{L\in\dtwo{n-2}}{R\in\dtwo{n-1}}} x^{1+\des(L)+1+\des(R)}\\
&= A_n(x) - \sum_{L\in\dtwo{n-2}} x^{1+\des(L)}\sum_{R\in\dtwo{n-1}} x^{1+\des(R)}\\
&= A_n(x) - A_{\lfloor (n-1)/2 \rfloor}(x) A_{{\lceil (n-1)/2\rceil}}.\qedhere
\end{align*}
\end{proof}

\begin{theorem}
\label{rl:nm2}
For all $n>3$,
\begin{align*}
W(\Tack{n}{n-3};x) &= A_n(x) - A_{\lfloor \frac{n-1}{2} \rfloor}(x) A_{{\lceil \frac{n-1}{2} \rceil}}\\
& \phantom{=} - \left\lfloor\dfrac{n+2}{2}\right\rfloor A_{\lfloor \frac{n-1}{2}\rfloor}(x) A_{\lfloor \frac{n}{2} \rfloor }(x)
- \left\lfloor\dfrac{n-1}{2}\right\rfloor A_{\lfloor \frac{n-2}{2}\rfloor}(x) A_{\lfloor \frac{n+1}{2} \rfloor }(x)\\
& \phantom{=}+A_{\lfloor \frac{n-1}{2} \rfloor}(x) D_{\lfloor \frac{n}{2} \rfloor}(x) - A_{\lfloor \frac{n-2}{2} \rfloor}(x) D_{\lfloor \frac{n+1}{2} \rfloor}(x)
\end{align*}
where 
\begin{align*}
D_n(x) &= \dfrac{1}{2} \sum_i \binom{n-1}{i} A_i(x) A_{n-1-i}(x).
\end{align*}
\end{theorem}

\begin{proof}
We wish to obtain an expression for
\begin{align*}
W(\Tack{n}{n-3};x) &= W(\Tack{n}{n-2};x) - \sum_{\mycom{\pi \in \sn}{\deg_{\T}(\pi)=n-2}} x^{1+\des(\pi)}.
\end{align*}
Theorem~\ref{rl:nm1} gives an expression for $W(\Tack{n}{n-2})$. 
We will now classify those permutations $\pi \in \sn$ for which $\deg_{\T}(\pi)=n-2$. 
Suppose that $\pi \in \sn$ with $\deg_{\T}(\pi)= n-2$. Then $\T^{n-2}(\pi)=\id$ and $\T^{n-3}(\pi) \neq \id$.
The contrapositive of Theorem~\ref{thm:nozigzag} tells us that $\pi$ must contain an $(n-3)$-zigzag.
(Theorem~\ref{thm:zigzag} does not assist us in this situation as it translates into a statement concerning a zigzag too large to fit into the permutation.)

Therefore will check all types of permutations $\pi$ in $\sn$ which contain an $(n-3)$-zigzag and classify those for which $\deg_{\T}(\pi)=n-2$.
In order to do this, we will condition on the lexicographically largest zigzag in each case so as to avoid overlapping cases.

An $(n-3)$-zigzag in a permutation $\pi$ is a strictly decreasing sequence $z=(z_0,\ldots, z_{n-2})$. Let us suppose that $z$ is maximal in the sense that
it is the lexicographically largest $(n-3)$-zigzag in $\pi$. Because $z$ contains $n-1$ elements out of a possible $n$ elements, and because it is strictly decreasing, it must be the sequence $(n,n-1,\ldots,1)$ with one element removed. We indicate this using a hat above the element that is removed:
$$z=(n,n-1,\ldots,\hat{i},\ldots,1).$$
To perform this classification, we must split the analysis into four distinct cases:
(a) $i=n$ (b) $i=n-1$ (c) $i=n-2$ and (d) $1\leq i \leq n-3$.

\noindent {\bf{Case (a) $i=n$:}}\\
In this case the $(n-3)$-zigzag is $(n-1,n-2,\ldots,1)$. 
Let $\pi'$ be $\pi$ with $n$ removed. Then $\pi'$ must be of the form $L(n-1)R$ where $L \in \Perm{\dtwo{n-3}}$ and $R \in \Perm{\dtwo{n-2}}$.
All that remains is to carefully check how the position of $n$ in $\pi$ changes the degree of $\pi$ with respect to $\T$.
We must be careful that by inserting $n$ in different places, the resulting
permutation does not contain an $(n-2)$-zigzag (as these were taken care of previously in Theorem~\ref{rl:nm1}).

Let $L_1L_2\in \Perm{\dtwo{n-3}}$ and $R_1R_2 \in \Perm{\dtwo{n-2}}$.
From Theorem~\ref{rl:nm1}, the permutation $\pi\in\sn$ must have one of two forms;
\begin{itemize}
\item $\pi=L_1nL_2(n-1)R_1R_2$: 
	First observe that if $L_2= \emptyset$, then $z'=(n,n-2,n-3,\ldots,1)$ will be an $(n-3)$-zigzag that is lexicographically greater than $z=(n-1,\ldots,1)$. This is forbidden so we must have $L_2\neq \emptyset$.
	Bearing this in mind, 
	\begin{align*}
	 \T(\pi) &= \T(L_1nL_2 (n-1)R_1R_2) \\ &= \T(R_1R_2)\T(L_2)(n-1)\T(L_1)n,\end{align*}
	where $\T(R_1R_2)$ ends in $(n-2)$. Suppose $\T(R_1R_2)=R_3(n-2)$ where $R_3\in\Perm{\dtwo{n-4}}$.
	Then
	\begin{align*}
	\T^2(\pi) &= \T(R_3(n-2)\T(L_2)(n-1)\T(L_1)n) \\
			&= \T^2(L_1)\T^2(L_2)\T(R_3)(n-2)(n-1)n.
	\end{align*}
	Notice that $\T^2(L_1)\T^2(L_2)\T(R_3) \in \sym_{n-3}$.
	Since $L_2 \neq \emptyset$, we have $\deg_{\T}(\pi)=n-2$ iff $\deg_{\T}(\T^2(\pi)) = n-4$, and this is true iff $n-3 \in \T^2(L_2)$.
	The set of all permutations satisfying $\deg_{\T}(\pi)=n-2$ in this case is
	\begin{align}
	\pi &= L(n-1)R\;\in\; \sn \nonumber\\
	L &\in \Perm{\{n,n-3,n-5,\ldots\} } \mbox{ and $n$ precedes $(n-3)$ in $L$}  \label{saturday}\\
	R &\in \Perm{\{n-2,n-4,\ldots\}}. \nonumber
	\end{align}
\item $\pi=L_1L_2(n-1)R_1nR_2$: 
	If $R_1 = \emptyset$ then $\pi$ will contain a $(n-2)$-zigzag that is lexicographically larger than $(n-1,\ldots,1)$. 
	This is forbidden and so we have $R_1\neq \emptyset$.
	With this in mind,
	$$\T(\pi) = \T(R_2) \T(R_1)\T(L_1L_2)(n-1)n.$$
	The permutation $\T(R_2) \T(R_1)\T(L_1L_2) \in \sym_{n-2}$ has degree $n-3$ iff $\T(R_2)\T(R_1)$ ends in $n-2$.
	This happens iff $(n-2) \in R_1$ (since we already showed that $R_1$ cannot be empty).
	The set of all permutations for which this is true is
	\begin{align}
	\pi &= L(n-1)R\;\in\;\sn \nonumber\\
	L &\in \Perm{\{n-3,n-5,\ldots \}} \label{sunday}\\
	R &\in \Perm{\{n,n-2,n-4,\ldots\} } \mbox{ and $n-2$ precedes $n$ in $R$.} \nonumber
	\end{align}
\end{itemize}

\noindent {\bf{Case (b) $i=n-1$:}}\\
In this case the maximal $(n-3)$-zigzag is $(n,n-2,n-3,\ldots ,1)$.
We will consider how the position of $(n-1)$ in $\pi$ affects the degree of the
resulting permutation. The permutation $\pi$ must have the form $L_1(n-1)L_2nR_1R_2$ or $L_1L_2nR_1(n-1)R_2$ 
where $L_1L_2\in\Perm{\dtwo{n-3}}$ and $R_1R_2 \in \Perm{\dtwo{n-2}}$
(otherwise there will be a lexicographically larger $(n-3)$-zigzag starting with $(n,n-1,\ldots)$ which is forbidden).
\begin{itemize}
\item $\pi=L_1(n-1)L_2nR_1R_2$:
	In this case $\T(\pi) = \T(R_1R_2) \T(L_2)\T(L_1) (n-1)n$. 
	Since $\T(R_1R_2)$ ends in $(n-2)$ we have that $\deg_{\T}(\T(R_1R_2) \T(L_2)\T(L_1)) = n-3$ 
	and so $\deg_{\T}(\pi) = n-2$. All permutations of this form have $\deg_{\T}(\pi) = n-2$.
	This collection is
	\begin{align}
	\pi&= LnR \; \in\; \sn \nonumber\\
	L&\in\Perm{\{n-1,n-3,\ldots\} }\label{monday}\\ 
	R&\in\Perm{\{n-2,n-4,\ldots\}}. \nonumber
	\end{align}
\item $L_1L_2nR_1(n-1)R_2$:
	This permutation contains another $(n-3)$-zigzag starting with $(n,n-1,n-3,n-4,\ldots)$. 
	This case is therefore not relevant since it violates the assumption that $(n,n-2,n-3,\ldots ,1)$ is maximal.
\end{itemize}

\noindent {\bf{Case (c) $i=n-2$:}}\\
The maximal $(n-3)$-zigzag in this case is $(n,n-1,n-3,\ldots ,1)$ so we will consider how the position of $(n-2)$ affects the degree of the
resulting permutation. The permutation $\pi$ is either of the form $L_1(n-2)L_2nR_1R_2$ or $L_1L_2nR_1(n-2)R_2$ where
$L_1L_2\in\Perm{\{n-3,n-5,\ldots\} }$ and $R_1R_2 \in \Perm{\{n-1,n-4,n-6,\ldots\}}$.
\begin{itemize}
\item $\pi=L_1(n-2)L_2nR_1R_2$: 
	This contains an $(n-3)$-zigzag that begins with $(n,n-1,n-2,n-4,\ldots)$ and therefore
	contradicts the assumption that $(n,n-1,n-3,\ldots ,1)$ is maximal amongst $(n-3)$-zigzags.
\item $\pi=L_1L_2nR_1(n-2)R_2$: 
	We have $$\T(\pi) = \T(L_1L_2nR_1(n-2)R_2) = \T(R_1(n-2)R_2) \T(L_1L_2) n.$$ 
	Consider the two sub-cases: $n-1 \in R_1$ and $n-1 \in R_2$.
	
	{\it{Sub-case $n-1 \in R_1$}}: If $n-1 \in R_1$ then $R_1=R_3(n-1)R_4$ and $R_2R_3R_4$ is a permutation of $\dtwo{n-4}$. 
	Consequently
	\begin{align*}
	\T(\pi)&=\T(R_3(n-1)R_4(n-2)R_2) \T(L_1L_2) n \\ &= \T(R_2)\T(R_4)(n-2)\T(R_3) (n-1) \T(L_1L_2) n 
	\end{align*}
	and $\T^2(\pi) = \T^2(L_1L_2) \T^2(R_3) \T(\T(R_2)\T(R_4)) (n-2)(n-1)n$. This is a permutation
	of the form $A(n-3)B(n-2)(n-1)n$ with $\T^2(L_1L_2)=A(n-3)$ where $A$ is permutation of $\dtwo{n-5}$ and 
	$B$ is a permutation of $\dtwo{n-4}$.
	We therefore have $\deg_{\T}(\T^2(\pi))=n-4$ which means $\deg_{\T}(\pi) = n-2$. 
	All permutations of this form have degree $n-2$.

	{\it{Sub-case $n-1 \in R_2$}}: If $n-1\in R_2$ then $R_2=R_3(n-1)R_4$ and $R_1R_3R_4$ is a permutation in $\Perm{\dtwo{n-4}}$.
	Consequently
	\begin{align*}
	\T(\pi) &=\T(R_1(n-2)R_3(n-1) R_4) \T(L_1L_2) n \\ 
	 	&= \T(R_4)\T(R_3)\T(R_1)(n-2)(n-1) \T(L_1L_2) n 
	\end{align*}
	and 
	\begin{align*}
		\T^2(\pi) &= \T(\T(R_4)\T(R_3)\T(R_1)(n-2)(n-1) \T(L_1L_2) n) \\
			  &= \T^2(L_1L_2) \T(\T(R_4)\T(R_3)\T(R_1)) (n-2)(n-1)  n .
		\end{align*}
	Since $\T^2(L_1L_2)$ ends in $n-3$ (we know this because $n-3 \in L_1L_2$ and so $n-3 \not\in R_1R_3R_4$) and 
	$$\T^2(L_1L_2) \T(\T(R_4)\T(R_3)\T(R_1))$$ is a permutation of the form $A(n-3)B \in \sym_{n-3}$
	with $\T^2(L_1L_2)=A(n-3)$
	where $A \in \Perm{\dtwo{n-5}}$ and $B\in\Perm{\dtwo{n-4}}$, we have $\deg_{\T}(\T^2(\pi)) = n-4$ which gives $\deg_{\T}(\pi)=n-2$.
	All permutations of this form have degree $n-2$.

	Combining both sub-cases yields the following classification of permutations for this case:
	\begin{align}
	\pi&=LnR\;\in\;\sn \nonumber \\
	L &\in \Perm{\{n-3,n-5,\ldots\}}    \label{thursday}\\
	R&\in\Perm{\{n-1,n-2,n-4,\ldots\}}. \nonumber
	\end{align}
\end{itemize}

\noindent {\bf{Case (d) $1\leq i \leq n-3$:}}\\
This case is in many ways similar to Case (c) so we only will highlight the main difference.
It turns out that in order to obtain $z$ maximal among $(n-3)$-zigzags,
the restriction in this case is that two consecutive elements of $\{1,\ldots,n-2\}$, let us call the pair $\overline{i}=\{i,i+1\}$, 
must both be either to the left or to the right of $n$ in the permutation $\pi$. 
Then the sequence of elements $n-1,n-2,\ldots ,i+2,\overline{i},i-1,\ldots 1$ alternate from right to left to right etc. of $n$ (for any value $1\leq i\leq n-3$).
The result depends, of course, on the parity of $i$ (i.e. whether we have $i$ and $i+1$ both to the left, or to the right, of $n$ in $\pi$).
\begin{itemize}
\item If $n-i$ is odd then $i$ and $i+1$ are both to the left of $n$ in $\pi$. 
The number of possible values of $i$ is $\lfloor (n-2)/2\rfloor$.
The sizes of $L$ and $R$ in this case are then $\lfloor n/2 \rfloor$ and $\lfloor (n-1)/2\rfloor$, respectively.
\item If $n-i$ is even then $i$ and $i+1$ are both to the right of $n$ in $\pi$. The number of possible values of $i$ is $\lfloor (n-3)/2 \rfloor$.
The sizes of $L$ and $R$ are $\lfloor (n-2)/2\rfloor$ and $\lfloor (n+1)/2\rfloor$, respectively.
\end{itemize}
The class of these permutations given by this classification and for which $\deg_{\T}(\pi)= n-2$ is
\begin{align}
\pi&= LnR \; \in\; \sn \nonumber \\
L &\in \Perm{\{n-2,n-4,\ldots ,i+1,i,i-2,i-4,\ldots\}} \label{friday}\\
R &\in \Perm{\{n-1,n-3,\ldots ,i+2,i-1,i-3,\ldots\} },  \nonumber \\
\noalign{for all appropriate $i \in [1,n/2]$ when $n-i$ is odd, and} \nonumber 
\pi&= LnR \; \in\; \sn \nonumber \\
L &\in \Perm{\{n-2,n-4,\ldots,i+2,i-1,i-3,\ldots\}} \label{noday}\\
R &\in \Perm{\{n-1,n-3,\ldots ,i+3,i+1,i,i-2,\ldots\} }, \nonumber\\
\noalign{for all appropriate $i \in [1,n/2]$ when $n-i$ is even.} \nonumber 
\end{align}

The permutations in Equations ~\ref{saturday}--\ref{noday} are all of the permutations $\pi\in \sn$ for which $\deg_{\T}(\pi) = n-2$.
We will use these to calculate the descent polynomial
\begin{align}\label{despol}
\sum_{\mycom{\pi\in\sn}{\deg_{\T}(\pi)=n-2}} x^{1+\des(\pi)}.
\end{align}
In some of the classifications of permutations, we see that the sum is over all permutations in a set except
that the largest value must always be to the left or right of the second largest value. (See for example Equations~\ref{saturday} and \ref{sunday}.)
To accommodate descent polynomials over these restricted sets we define:
\begin{align*}
D_n(x) &= \sum_{\mycom{\pi \in \sn}{{\pi^{-1}(n) > \pi^{-1}(n-1)}}} x^{1+\des(\pi)} \quad \mbox{ and } \quad
L_n(x) = \sum_{\mycom{\pi \in \sn}{{\pi^{-1}(n) < \pi^{-1}(n-1)}}} x^{1+\des(\pi)}.
\end{align*}
Clearly $D_n(x)+L_n(x)=A_n(x)$.
In order to give an expression for $D_n(x)$ in terms of polynomials that we know, we do as follows.
The polynomial $D_n(x)$ is the Eulerian distribution on length $n$ permutations for which $n$ is to the right of $n-1$ in the permutation.

Split the sum over those permutations $\pi\in\sn$ for which (i) $\pi_n=n$ and (ii) $\pi_j=n$ for some $1\leq j \leq n-1$.
        In the second case, there are $\binom{n-2}{j-2}$ ways of choosing the elements which precede $n$ in $\pi$, because $n-1$ must be one of them.
        Let us write $\pi = \pi ' n \pi''$.
        So we have
        \begin{eqnarray*}
        D_n(x) &=& \sum_{\mycom{\pi \in \sn}{\pi_n=n}} x^{1+\des(\pi)} +
                \sum_{j=2}^{n-1}
                \sum_{\mycom{A\subseteq \{1,\ldots,n-2\}}{|A|=j-2}}
                \sum_{\mycom{\pi' \in \LittlePerm(A\cup \{n-1\})}{\pi '' \in \LittlePerm(\{1,\ldots,n-2\}-A)}}
                        x^{1+\des(\pi=\pi' n \pi'')}
        \end{eqnarray*}
        where $\pi = \pi ' n \pi''$ and $\LittlePerm(X)$ is the set of all permutations of the elements in $X$. Clearly the first sum is $A_{n-1}(x)$.
        Since $\des(\pi)=\des(\pi ' ) + 1 +d(\pi'')$ for all $2\leq j \leq n-1$, we have $x^{1+\des(\pi)} = x^{1+\des(\pi ' )} x^{1+\des(\pi '')}$.
        Replacing this in the second summation, we find
        \begin{eqnarray*}
        D_n(x) &=& A_{n-1}(x) +
        \sum_{j=2}^{n-1} \binom{n-2}{j-2} A_{j-1}(x) A_{n-j}(x)
        \end{eqnarray*}
        Now
        \begin{eqnarray*}
        \lefteqn{\sum_{j=2}^{n-1} \binom{n-2}{j-2} A_{j-1}(x) A_{n-j}(x)} \\
        &=& \sum_{j=1}^{n-2} \binom{n-2}{j-1} A_{j}(x) A_{n-1-j}(x) \\
        &=& \dfrac{1}{2} \left( \sum_{j=1}^{n-2} \binom{n-2}{j-1} A_j(x) A_{n-1-j}(x) + \sum_{i=1}^{n-2} \binom{n-2}{n-2-i} A_{n-1-i}(x) A_i(x)   \right) \\
        &=& \dfrac{1}{2} \left( \sum_{j=1}^{n-2} \binom{n-2}{j-1} A_j(x) A_{n-1-j}(x) + \sum_{j=1}^{n-2} \binom{n-2}{j} A_{n-1-j}(x) A_j(x)   \right) \\
        &=& \dfrac{1}{2} \sum_{j=1}^{n-2} \binom{n-1}{j} A_j(x) A_{n-1-j}(x).
        \end{eqnarray*}
        Since $A_{n-1}(x) = \frac{1}{2} \left( A_0(x)A_{n-1}(x) + A_{n-1}(x) A_{0}(x)   \right)$, we have
\begin{align} \label{didentity}
D_n(x) &= \dfrac{1}{2} \sum_i \binom{n-1}{i} A_i(x) A_{n-1-i}(x).
\end{align}
The sums over all permutations in each of Equations ~\ref{saturday}--\ref{noday} are given in the table in Figure~\ref{fig:one}, and their total noted in the 
final row. 

\begin{figure}
$$\begin{array}{|c|c|} \hline
\mbox{A} & \ds\sum_{\pi\in \mathrm{Equation\; A}  } x^{1+\des(\pi)} \\ \hline \hline
\mbox{(\ref{saturday})} & L_{\lfloor \frac{n}{2}\rfloor}(x) A_{\lfloor \frac{n-1}{2}\rfloor }(x) \\[0.5em]
\mbox{(\ref{sunday})} & A_{\lfloor \frac{n-2}{2}\rfloor}(x) D_{\lfloor \frac{n+1}{2}\rfloor }(x)  \\[0.5em]
\mbox{(\ref{monday})} & A_{\lfloor \frac{n}{2}\rfloor}(x) A_{\lfloor \frac{n-1}{2} \rfloor }(x) \\[0.5em]
\mbox{(\ref{thursday})} & A_{\lfloor \frac{n-2}{2}\rfloor}(x) A_{\lfloor \frac{n+1}{2} \rfloor }(x) \\[0.5em]
\mbox{(\ref{friday})} & \left\lfloor\frac{n-2}{2}\right\rfloor A_{\lfloor \frac{n}{2}\rfloor}(x) A_{\lfloor \frac{n-1}{2} \rfloor }(x) \\[0.5em]
\mbox{(\ref{noday})} & \left\lfloor\frac{n-3}{2}\right\rfloor A_{\lfloor \frac{n-2}{2}\rfloor}(x) A_{\lfloor \frac{n+1}{2} \rfloor }(x) \\[0.5em] \hline\hline
\mbox{Total} & \left\lfloor\frac{n}{2}\right\rfloor A_{\lfloor \frac{n}{2}\rfloor}(x) A_{\lfloor \frac{n-1}{2} \rfloor }(x)
	+\left\lfloor\frac{n-1}{2}\right\rfloor A_{\lfloor \frac{n-2}{2}\rfloor}(x) A_{\lfloor \frac{n+1}{2} \rfloor }(x)	 \\[0.5em]
\star \star \star & + L_{\lfloor \frac{n}{2}\rfloor}(x) A_{\lfloor \frac{n-1}{2}\rfloor }(x) + A_{\lfloor \frac{n-2}{2}\rfloor}(x) D_{\lfloor \frac{n+1}{2}\rfloor }(x) \\ \hline
\end{array}$$
\caption{Contributions to the descent polynomial from different cases.\label{fig:one}}
\end{figure}
We thus have
\begin{align*}
\lefteqn{W(\Tack{n}{n-3};x)}\\
&= W(\Tack{n}{n-2};x) - (\star\star\star) \\
&= A_n(x) - A_{\lfloor \frac{n-1}{2} \rfloor}(x) A_{{\lceil \frac{n-1}{2}\rceil}}  - \left\lfloor\frac{n}{2}\right\rfloor A_{\lfloor \frac{n}{2}\rfloor}(x) A_{\lfloor \frac{n-1}{2} \rfloor }(x)
-\left\lfloor\frac{n-1}{2}\right\rfloor A_{\lfloor \frac{n-2}{2}\rfloor}(x) A_{\lfloor \frac{n+1}{2} \rfloor }(x)\\
&\phantom{= } - L_{\lfloor \frac{n}{2}\rfloor}(x) A_{\lfloor \frac{n-1}{2}\rfloor }(x) - A_{\lfloor \frac{n-2}{2}\rfloor}(x) D_{\lfloor \frac{n+1}{2}\rfloor }(x).
\end{align*}
The statement of the theorem follows by substituting $L_{\lfloor \frac{n}{2}\rfloor}(x)=A_{\lfloor \frac{n}{2}\rfloor}(x) - D_{\lfloor \frac{n}{2}\rfloor}(x).$
\end{proof}

\section{Steingr\'imsson's conjecture}
In this section we will use the results of Section 3 to prove Steingr\'imsson's conjecture (Conjecture~\ref{econj}) in some special cases.

\begin{theorem}
Let $n\geq 3$ be fixed. Then $|\Stack{n}{k}|\leq |\Tack{n}{k}|$ for $k=0,1,2,n-3,n-2$, and $n-1$.
This inequality is strict for all of these pairs $(n,k)$ that satisfy $2<k<n-1$
\end{theorem}

\begin{proof}
Since $\Stack{n}{0} = \Tack{n}{0} = \{\id\}$, the stated inequality is true for $k=0$.
Also, $\Stack{n}{n-1}=\Tack{n}{n-1} = \sn$, so the inequality is true for $k=n-1$.
The cases for $k=1$ and $k=2$ follow from Theorems~\ref{onestack} and \ref{twostack}, respectively, by setting $x=1$.
The outstanding cases are $k=n-3$ and $k=n-2$ which require more work.

First let us consider $k=n-2$.
West~\cite{west.phd} enumerated the set $\Stack{n}{n-2}$:
$$|\Stack{n}{n-2}|=n!-(n-2)!$$
If we set $x=1$ in Theorem~\ref{rl:nm1} then we find that 
$$|\Tack{n}{n-2}| = n!- \left\lfloor \frac{n-1}{2} \right\rfloor ! \left \lceil \frac{n-1}{2}\right \rceil !$$
The condition $|\Stack{n}{n-2}|\leq |\Tack{n}{n-2}|$ is equivalent to 
\begin{align}\label{labelme}
\left\lfloor \frac{n-1}{2} \right\rfloor ! \left \lceil \frac{n-1}{2}\right \rceil ! &\leq (n-2)!
\end{align}
\begin{itemize}
\item If $n=2m+1$ then this translates to $m! m! \leq (2m-1)!$, which is equivalent to $\binom{2m}{1} \leq \binom{2m}{m}$, and this is true for all $m\geq 1$. This is a strict inequality for $m\geq 2$.
\item If $n=2m+2$ then this translates to $m! (m+1)! \leq (2m)!$, which is equivalent to 
$\binom{2m+1}{1} \leq \binom{2m+1}{m}$, and this is true for all $m\geq 0$. This is a strict inequality for $m\geq 2$.
\end{itemize}
The inequality \ref{labelme} therefore holds true for all $n\geq 2$ and so the conjecture is true for $k=n-2$.
Furthermore, the inequality is strict when $n\geq 5$.

We now consider the final case $k=n-3$.
Using West~\cite[Ex. 4.2.15 and Thm 4.2.17]{west.phd} we have
$$|\Stack{n}{n-3}| = \dfrac{(n-3)!}{2} \left( 2n^3-6n^2-5n+16 \right).$$
To get an expression for $|\Tack{n}{n-3}|$, set $x=1$ in Theorem~\ref{rl:nm2} to get
\begin{align}
|\Tack{n}{n-3}| =& n! - \floor{n-1}{2}! \ceil{n-1}{2}!  - \floor{n+2}{2} \floor{n-1}{2}! \floor{n}{2}!\nonumber \\
	&- \floor{n-1}{2} \floor{n-2}{2}! \floor{n+1}{2}!
	+\tfrac{1}{2} \floor{n-1}{2}! \floor{n}{2}!
	-\tfrac{1}{2} \floor{n-2}{2}! \floor{n+1}{2}!\label{rara}
\end{align}
\begin{itemize}
\item If $n=2m+1$ then $|\Tack{n}{n-3}| = (2m+1)! - \tfrac{4m^2+6m+1}{2} (m-1)!m!$ and $|\Stack{n}{n-3}|$ $=$ $\tfrac{(2m-2)!}{2} \left(16m^3-22m+7\right)$.
We find that $|\Stack{n}{n-3}| \leq |\Tack{n}{n-3}|$ iff
\begin{align}\label{ineqfirst}
\binom{2m-2}{m-1} &\geq \frac{4m^3+6m^2+m}{18m-7}.
\end{align}
This inequality is true for $m=0$. Suppose it to be true for $m=t$. Then
\begin{align*}
\binom{2(t+1)-2}{(t+1)-1} &= \frac{2(2t-1)}{t} \binom{2t-2}{t-1}  \geq \frac{2(2t-1)}{t} \frac{4t^3+6t^2+t}{18t-7},
\end{align*}
by the induction hypothesis. The value on the right hand side satisfies 
the inequality
\begin{align*}
\dfrac{2(2t-1)(4t^2+6t+1)}{18t-7} & \geq \dfrac{4(t+1)^3+6(t+1)^2+(t+1)}{18(t+1)-7},
\end{align*}
this being due to the fact that 
$216t^4 + 168t^3 - 292t^2 - 147t + 55\geq 0$ for all $t\geq 1$. (The equation $216t^4 + 168t^3 - 292t^2 - 147t + 55=0$ has only real roots, of which $t=1$ is the largest.)
Combining both of these inequalities shows that inequality \ref{ineqfirst} is true for $m=t+1$. By the principle of induction the inequality \ref{ineqfirst} is therefore true for all $m\geq 0$.
\item If $n=2m+2$ then $|\Tack{n}{n-3}| = (2m+2)! - (2m+3) m!(m+1)!$ and $|\Stack{n}{n-3}|=\tfrac{(2m-1)!}{2}\left(16m^3+24m^2-10m-2\right)$.
We find that $|\Stack{n}{n-3}| \leq |\Tack{n}{n-3}|$ iff
\begin{align}\label{ineqsecond}
\dfrac{(2m-1)!}{m!(m+1)!} &\geq \frac{2m+3}{9m+1}.
\end{align}
This inequality is true for $m=1$. Suppose it to be true for $m=t$.
Then 
\begin{align*}
\dfrac{(2(t+1)-1)!}{(t+1)!(t+2)!} &= \dfrac{2t(2t+1)}{(t+1)(t+2)} \dfrac{(2t-1)!}{t!(t+1)!} \geq \dfrac{2t(2t+1)}{(t+1)(t+2)} \dfrac{2t+3}{9t+1},
\end{align*}
by the induction hypothesis. 
By noticing that $54t^4 + 123t^3 + 32t^2 - 49t - 10 \geq 0$ for all $t\geq 1$ (the largest root of $54t^4 + 123t^3 + 32t^2 - 49t - 10=0$ is between 0.5 and 0.6),
the value on the right hand side satisfies the inequality
\begin{align*}
\dfrac{2t(2t+1)(2t+3)}{(t+1)(t+2)(9t+1)} & \geq \dfrac{2(t+1)+3}{9(t+1)+1}.
\end{align*}
Combining both of these inequalities we find that
\begin{align}
\dfrac{(2(t+1)-1)!}{(t+1)!(t+2)!} &\geq \frac{2(t+1)+3}{9(t+1)+1}
\end{align}
and inequality \ref{ineqsecond} is true for $m=t+1$.
By the principle of induction the inequality \ref{ineqsecond} is therefore true for all $m\geq 1$.
\end{itemize}
In both cases the inequality is strict for $n> 5$.\qedhere
\end{proof}
Equation \ref{rara} admits a reduction to the more compact form:
\begin{align}
|\Tack{n}{n-3}| &= n! - \floor{n-1}{2}! \ceil{n-1}{2}!  - \tfrac{n^2}{2} \floor{n-1}{2}! \floor{n-2}{2}!
\end{align}

\section{The descent polynomial of $t$-revstack sortable permutations}
Let $V_t(n,i)$ and $W_t(n,i)$ be the sets of permutations in $\sn$ having $i$ descents which are $t$-\revstack\ sortable and $t$-stack sortable, respectively.
Let us write $v_t(n,i)=|V_t(n,i)|$ and $w_t(n,i)=|W_t(n,i)|$ so that
\begin{align*}
W(\Tack{n}{t};x) &= \sum_{i} v_t(n,i) x^{i+1} \\
W(\Stack{n}{t};x) &= \sum_{i} w_t(n,i) x^{i+1}.
\end{align*}
B\'ona\cite{bona.1,bona.2} 
proved symmetry and unimodality of the numbers $w_t(n,i)$ with respect to the descent parameter, i.e.
\begin{align*}
w_t(n,i) &= w_t(n,n-1-i) \mbox{ for all }0\leq i \leq n-1 \mbox{, and}\\
w_t(n,0) &\leq w_t(n,1) \leq \ldots \leq w_t(n,\lfloor (n-1)/2\rfloor) \geq \ldots \geq w_t(n,n-1).
\end{align*}
In this section we will do the same for the numbers $v_t(n,i)$.
Symmetry is proven by showing that the reflection operation on permutations commutes with B\'ona's~\cite{bona.1} 
{\it{duality map}}.
Unimodality is proven by using the same argument of B\'ona~\cite{bona.1,bona.2} and showing that his function $z:W_t(n,i) \to W_t(n,i+1)$ 
is injective for all $0\leq i \leq \floor{n-3}{2}$ and preserves $t$-\revstack\ sortability.

\begin{theorem}
$v_t(n,i) = v_{t}(n,n-1-i)$.
\end{theorem}

\begin{proof}
B\'ona's duality map~\cite[Section 2]{bona.1} $f:\sn \to \sn$ is defined in the following recursive way. 
Note that permutations are written as words, and the definition is necessarily a definition on words.
The value $n$ stands for the unique largest value in a word.
\begin{itemize}
\item $f(\epsilon) = \epsilon$
\item $f(x)=x$ where $x$ is a word of length 1
\item $f(LnR)=f(L)nf(R)$ if neither $L$ nor $R$ is empty
\item $f(Ln)=nf(L)$
\item $f(nR)=f(R)n$.
\end{itemize}
Let us recall B\'ona's method of proof of symmetry. 
He showed~\cite[Lemma 2.6]{bona.1} that $f$ preserves the $t$-stack sortable property: 
\begin{align}\label{bonapreserve}
\S(\pi)&=\S(f(\pi))
\end{align}
from which he concludes that $\pi$ is $t$-stack sortable iff $f(\pi)$ 
is $t$-stack sortable. 
He also showed~\cite[Proposition 2.5]{bona.1} that 
\begin{align}\label{bonades}
\des(\pi) + \des(f(\pi)) &= n-1
\end{align}
for all $\pi \in \sn$.
Combining these two observations: applying $f$ to permutations in $\sn$ that are $t$-stack sortable gives another permutation that is $t$-stack sortable, but which has $n-1-\des(\pi)$ descents, i.e. there is symmetry in the descent statistic.

The reverse operation $\rev:\sn \to \sn$ is defined recursively via
\begin{itemize}
\item $\rev(\epsilon) = \epsilon$
\item $\rev(LnR) = \rev(R)n \rev(L)$ where $n$ is the largest value in the word $LnR$.
\end{itemize}
It is easy (by induction on the length of the word) to see that $f \circ \rev = \rev \circ f = g$ where $g$ is defined as follows:
\begin{itemize}
\item $g(\epsilon)=\epsilon$
\item $g(x)=x$ where $x$ is a word of length 1
\item $g(LnR)=g(R)ng(L)$ if neither $L$ nor $R$ is empty
\item $g(Ln)=g(L)n$
\item $g(nR)=ng(R)$.
\end{itemize}
The operator $f$ preserves the $t$-\revstack\ sortable property since
\begin{align*}
\T(\pi) &= \S(\rev(\pi))\\
&= \S(f(\rev(\pi)))\\
\noalign{by applying B\'ona's equation \ref{bonapreserve} above,}
&= \S(\rev(f(\pi)))\\
\noalign{since $f$ commutes with $\rev$,}
&= \T(f(\pi)).
\end{align*}
This fact, in conjunction with equation~\ref{bonades}, gives the stated result.
\end{proof}

\begin{theorem}
The sequence of numbers $(v_t(n,0),v_t(n,1),\ldots,v_t(n,n-1))$ is unimodal.
\end{theorem}

\newcommand\Tree{\mathsf{Tree}}
\begin{proof}
This proof is a slight modification of B\'ona's unimodality proof of the number sequence $(w_t(n,i))_{i=0}^{n-1}$ (see \cite[Section 3]{bona.1} and \cite{bona.2}).
Let $\sn^{\des=i}$ be the set of permutations in $\sn$ that have $i$ descents.
Let $T(n,i)$ be the set of decreasing binary trees on $n$ vertices whose label set is $\{1,\ldots,n\}$ and which have $i$ right edges.

The sets $\sn^{\des=i}$ and $T(n,i)$ are in one-to-one correspondence, and we write $\Tree(\pi)$ for the tree in $T(n,i)$ that corresponds to $\pi \in \sn^{\des=i}$ using the following rule:
{{$x$ is a left (resp. right) child of $y$ in $\Tree(\pi)$ iff $x$ is the largest value to the left (resp. right) of $y$ in $\pi$ and which is less than it.}}
(One easily recovers the permutation $\pi$ from $\Tree(\pi)$ by reading its labels using in-order traversal, and $\des(\pi)$ equals the number of right edges in $\Tree(\pi)$.)

B\'ona proved unimodality by exhibiting an function $h':T(n,i) \to T(n,i+1)$ that is injective for 
all $i\leq \floor{n-3}{2}$ and preserves the $t$-stack sortable property.
Due to the bijective correspondence between $\sn^{\des=i}$ and $T(n,i)$ (outlined in the previous paragraph), 
the injective function $h'$ is equivalent to a function 
$h: W_t(n,i) \to W_t(n,i+1)$ that is injective for all $0\leq i \leq \floor{n-3}{2}$ and for all $t$.

Let us describe the function $h$ by way of an example, and explain why it also preserves the property of $t$-revstack-sortability.
The function $h$ operates on a permutation $\pi$ by looking at its tree $\Tree(\pi)$.
Suppose $\pi = (8,7,9,4,6,1,10,2,3,5,11)$. Then $\Tree(\pi)$ is given in Figure~\ref{d:One}.

\begin{figure}[h!]
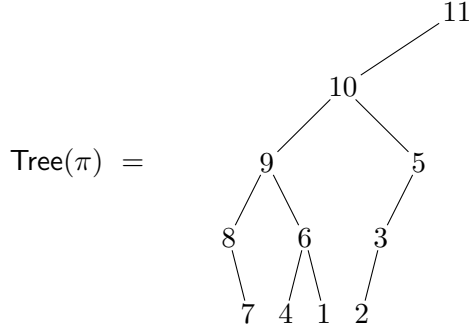

\begin{center}
{\diagramOne}
\end{center}
\caption{The decreasing binary tree corresponding to $\pi=(8,7,9,4,6,1,10,2,3,5,11).$\label{d:One}}
\end{figure}

Index the vertices of the tree from bottom to top, and within every level from left to right. 
This starts by labeling the bottom-most vertex on the left hand side $v_1$. 
If there is a vertex on the same level and to its right then label it vertex $v_2$, 
but otherwise go up one level and label the leftmost vertex $v_2$, and so forth. 
The labeling of $\Tree(\pi)$ in Figure~\ref{d:One} is given in Figure~\ref{d:Two}.

\begin{figure}[h!]
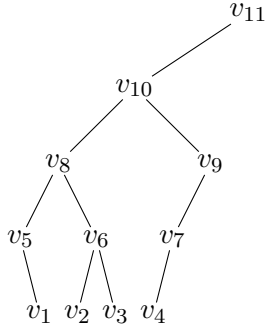

\begin{center}
{\diagramTwo}
\end{center}
\caption{The labeling of the tree from Figure~\ref{d:One}.\label{d:Two}}
\end{figure}

Let $T_i$ be the subgraph of $\Tree(\pi)$ when restricted to the vertices $\{v_1,\ldots,v_i\}$.
$T_7$ is shown in Figure~\ref{d:Three}.

\begin{figure}[h!]
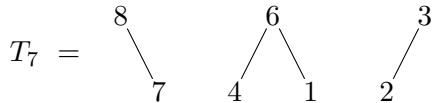

\begin{center}
{\diagramThree}
\end{center}
\caption{$T_7$, the restriction of $\Tree(\pi)$ to the vertices $\{v_1,\ldots,v_7\}$.\label{d:Three}}
\end{figure}

Determine the minimum value of $i$ such that the number of left edges in $T_i$ is precisely one more than the number of right edges in $T_i$. 
(B\'{o}na proved the existence of such an $i$ using a continuity argument~\cite{bona.2}.)
In our example this minimum value is $i=9$.

Next we look at the vertices in $\Tree(\pi)$ that are both part of $T_i$ and have only either a left or a right child (but not both).
If such a vertex has a left child, then move it to be a right child instead. If such a vertex has a right child, then move it to be a left child instead.
Let the resulting tree be $h'(\Tree(\pi))$.

For example, the nodes in $\Tree(\pi)$ that are part of $T_9$ and which have only left or right children (but not both) are 
$v_5$, $v_7$ and $v_9$. Consequently $h'(\Tree(\pi))$ is illustrated in Figure~\ref{d:Four}.

\begin{figure}[h!]
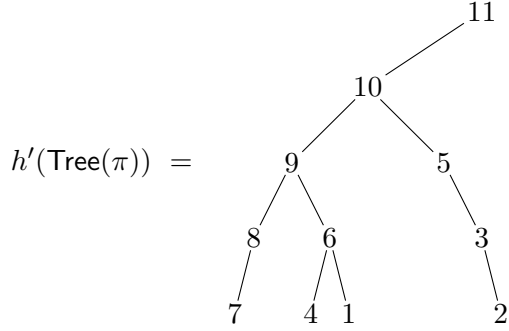

\begin{center}
{\diagramFour}
\end{center}
\caption{The outcome of applying $h'$ to $\Tree(\pi)$.\label{d:Four}}
\end{figure}

The tree $h'(\Tree(\pi))$ has exactly one more right edge than $\Tree(\pi)$.
This map $h':T(n,i) \to T(n,i+1)$ was proven to be injective for $i \leq \floor{n-3}{2}$ (see B\'ona\cite{bona.2}).
The permutation $h(\pi)$ is the unique permutation corresponding to $h'(\Tree(\pi))$. For our example we have
$$h(\pi) = (7,8,9,4,6,1,10,5,3,2,11).$$
To summarise: the function $h:\sn^{\des=i} \to \sn^{\des=i+1}$ is injective for all $0\leq i \leq \floor{n-3}{2}$. 

B\'ona then showed that $h$ preserves the $t$-stack sortable property: $\S(\pi) = \S(h(\pi))$.
This was done by noting that $\S(\pi)$ is given by reading the tree $\Tree(\pi)$ using post-order traversal.
The post-order reading of the tree is the same even after the application of the operator $h'$ on $\Tree(\pi)$ because $h'$ only changes the direction of subtrees that are the only-children of nodes, and therefore preserves the post-order reading.

The same argument holds in our case, and all we need to show in order to prove the theorem is that $h$ preserves the $t$-\revstack\ sortable property: $\T(\pi)=\T(h(\pi))$. 
To do this, let us define the {\it{rpostorder}} traversal of a tree as the list one obtains by first listing the rpostorder of its right subtree, then the rpostorder of the left subtree, and then finally listing the label of the root node. 
The rpostorder of the tree $\Tree(\pi)$ is then $\T(\pi)$.
Since $h'$ only changes the orientation of subtrees that are only-children of nodes, the rpostorder reading of $h'(\Tree(\pi))$ is the same as the rpostorder reading of $\Tree(\pi)$.
\end{proof}

The coefficients of the two smallest and two largest powers of $x$ in $W(\Tack{n}{t};x)$ and $W(\Stack{n}{t};x)$ are always the same. This follows from symmetry and using the following results. Note that the coefficient of $x$ in both is always unity because $\Tack{n}{0}=\Stack{n}{0} = \{\id\}$ which means $v_t(n,0)=w_t(n,0)=1$.

\begin{proposition}
$v_t(n,1) = w_t(n,1)$.
\end{proposition}

\begin{proof}
The set $\{\pi \in \sn: \T(\pi)=\id \mbox{ and } \des(\pi)=1\} = \{\pi\in \sn: \S(\rev(\pi))=\id \mbox{ and } \des(\pi)=1\}$. 
Since $\des(\rev(\pi))+\des(\pi)=n-1$ we have 
$|\{\pi\in \sn: \S(\rev(\pi))=\id \mbox{ and } \des(\pi)=1\}|=|\{\pi\in \sn: \S(\pi)=\id \mbox{ and } \des(\rev(\pi))=n-2\}|$.
This last set is the same as $\{\pi\in \sn: S(\pi)=\id \mbox{ and } \des(\pi)=1\}$  since $\des(\pi)=n-1-\des(\rev(\pi))$, hence the result.
\end{proof}

Proving real-rootedness and log-concavity of the coefficient sequences of the descent polynomials for $t$-stack sortable permutations remains an unsolved problem.
We present companion conjectures for revstack sort that have been verified
for all $0\leq t <n\leq 10$. All such polynomials, along with their roots, are listed in the appendix.
\begin{conjecture}  \label{realrooted}
For all $n\geq 1$, the polynomial $W(\Tack{n}{t};x)$ has only real non-positive roots.
\end{conjecture}
\begin{conjecture}  \label{logconcave}
For all $n\geq 1$, the sequence of coefficients of the polynomial\\ $W(\Tack{n}{t};x)$ is log-concave.
\end{conjecture}
Note that a proof of Conjecture~\ref{realrooted} would imply Conjecture~\ref{logconcave}.
\begin{conjecture} For all $n \geq 1$,
$W(\Tack{n}{n-2};x)$ has $n$ distinct real roots: $$r^{(n)}_{n-1} < r^{(n)}_{n-2} < \cdots < r^{(n)}_{1} < r^{(n)}_{0}=0$$
and these roots are related via
$$r^{(n+1)}_{n} < r^{(n)}_{n-1} < r^{(n+1)}_{n-1} < r^{(n)}_{n-2} < r^{(n+1)}_{n-2} \cdots < r^{(n+1)}_{2} < r^{(n)}_{1} < r^{(n+1)}_{1}< r^{(n)}_{0}=0=r^{(n+1)}_{0}.$$
\end{conjecture}

\section{Concluding Remarks }
\begin{enumerate}
\item 
\'Ulfarsson~\cite{henning} recently classified 3-stack sortable permutations in terms of a new type of permutation pattern that he called decorated patterns.
Can a similar analysis be performed for the case of 3-revstack sortable permutations? 
The stack sort operator preserves a certain left-to-rightness, whereas this is not the case for revstack.
\item In the paper \cite{annals} the author (in collaboration with two others) classified and enumerated $(n-4)$-stack sortable permutations. It was proven that
$|\Stack{n}{n-4}| = (n-4)!(3n^4 - 18n^3-4n^2+158n-192)/3$. Determining those permutations that are $(n-4)$-revstack sortable would appear to be tractable but most likely hard. Enumerating this class would allow for Steingr\'imsson's conjecture to be proven for $t=n-4$.
\item 
How many permutations $\pi \in \sn$ contain no $t$-zigzag? 
How many permutations $\pi \in \sn$ contain no uninterrupted $t$-zigzag? 
Answers to these questions will provide bounds on the number of $t$-revstack sortable permutations.
\end{enumerate}

\section*{Appendix}
\newcommand{\haha}{\hrule \noindent}
\newcommand{\shanda}[4]{\normalsize $W(\Tack{#1}{#2};x)=$\\ $#3$.\\ \phantom{a} \hfill\footnotesize Roots: $\left(#4\right)$\\}
\newcommand{\sshanda}[4]{\normalsize $W(\Tack{#1}{#2};x) = #3$.\\ \phantom{aaaa} \hfill\footnotesize Roots: $\left(#4\right)$\\}
\newcommand{\ssshanda}[4]{\normalsize $W(\Tack{#1}{#2};x) = #3$. \phantom{aaaa} \hfill\footnotesize Roots: $\left(#4\right)$\\}
\haha
\noindent
\ssshanda{ 1 }{ 0 }{ x }{ [0] }
\haha
\ssshanda{ 2 }{ 0 }{ x }{ [0] }
\ssshanda{ 2 }{ 1 }{ x^2 + x }{ [-1, 0] }
\haha
\ssshanda{ 3 }{ 0 }{ x }{ [0] }
\ssshanda{ 3 }{ 1 }{ x^3 + 3x^2 + x }{ [-2.61803, -0.38197, 0] }
\ssshanda{ 3 }{ 2 }{ x^3 + 4x^2 + x }{ [-3.73205, -0.26795, 0] }
\haha
\ssshanda{ 4 }{ 0 }{ x }{ [0] }
\ssshanda{ 4 }{ 1 }{ x^4 + 6x^3 + 6x^2 + x }{ [  -4.79129, -1, -0.20871, 0] }
\ssshanda{ 4 }{ 2 }{ x^4 + 10x^3 + 10x^2 + x }{ [-8.88748, -1, -0.11252, 0] }
\ssshanda{ 4 }{ 3 }{ x^4 + 11x^3 + 11x^2 + x }{ [-9.89898, -1, -0.10102, 0] }
\haha
\ssshanda{ 5 }{ 0 }{ x }{ [0] }
\sshanda{ 5 }{ 1 }{ x^5 + 10x^4 + 20x^3 + 10x^2 + x }{ [-7.51264,   -1.79811, -0.55614, -0.13311, 0] }
\sshanda{ 5 }{ 2 }{ x^5 + 20x^4 + 49x^3 + 20x^2 + x }{ [-17.22204,  -2.28160, -0.43829, -0.05807, 0] }
\sshanda{ 5 }{ 3 }{ x^5 + 25x^4 + 64x^3 + 25x^2 + x }{ [-22.163124, -2.36978, -0.42198, -0.04512, 0] }
\sshanda{ 5 }{ 4 }{ x^5 + 26x^4 + 66x^3 + 26x^2 + x }{ [-23.203854, -2.32247, -0.43058, -0.04310, 0] }
\haha
\ssshanda{ 6 }{ 0 }{ x }{ [0] }
\sshanda{ 6 }{ 1 }{ x^6 + 15x^5 + 50x^4 + 50x^3 + 15x^2 + x }{ [-10.78022, -2.76541, -1, -0.36161, -0.09276, 0] }
\sshanda{ 6 }{ 2 }{ x^6 + 35x^5 + 168x^4 + 168x^3 + 35x^2 + x }{ [-29.49606, -4.23384, -1, -0.23619, -0.03390, 0] }
\sshanda{ 6 }{ 3 }{ x^6 + 50x^5 + 267x^4 + 267x^3 + 50x^2 + x }{ [-44.07961, -4.68422, -1, -0.21348, -0.02269, 0] }
\sshanda{ 6 }{ 4 }{ x^6 + 56x^5 + 297x^4 + 297x^3 + 56x^2 + x }{ [-50.20122, -4.55954, -1, -0.21932, -0.01992, 0] }
\sshanda{ 6 }{ 5 }{ x^6 + 57x^5 + 302x^4 + 302x^3 + 57x^2 + x }{ [-51.21838, -4.54193, -1, -0.22017, -0.01952, 0] }
\haha
\ssshanda{ 7 }{ 0 }{ x }{ [0] }
\sshanda{ 7 }{ 1 }{ x^7 + 21x^6 + 105x^5 + 175x^4 + 105x^3 + 21x^2 + x }{ [-14.59334, -3.89836, -1.52940, -0.65385, -0.25652, -0.06852, 0] }
\sshanda{ 7 }{ 2 }{ x^7 + 56x^6 + 462x^5 + 900x^4 + 462x^3 + 56x^2 + x }{ [-46.47035, -6.98835, -1.83034, -0.54635, -0.14310, -0.02152, 0] }
\sshanda{ 7 }{ 3 }{ x^7 + 91x^6 + 898x^5 + 1834x^4 + 898x^3 + 91x^2 + x }{ [-80.06899, -8.36871, -1.90552, -0.52479, -0.11949, -0.01249, 0] }
\sshanda{ 7 }{ 4 }{ x^7 + 112x^6 + 1113x^5 + 2258x^4 + 1113x^3 + 112x^2 + x }{ [-101.22387, -8.24273, -1.86641, -0.53579, -0.12132, -0.00988, 0] }
\sshanda{ 7 }{ 5 }{ x^7 + 119x^6 + 1183x^5 + 2398x^4 + 1183x^3 + 119x^2 + x }{ [-108.27803, -8.19127, -1.86246, -0.53692, -0.12208, -0.00924, 0] }
\sshanda{ 7 }{ 6 }{ x^7 + 120x^6 + 1191x^5 + 2416x^4 + 1191x^3 + 120x^2 + x }{ [-109.30521, -8.15963, -1.86818, -0.53528, -0.12255, -0.00915, 0] }
\haha
\ssshanda{ 8 }{ 0 }{ x }{ [0] }
\sshanda{ 8 }{ 1 }{ x^8 + 28x^7 + 196x^6 + 490x^5 + 490x^4 + 196x^3 + 28x^2 + x }{ [-18.95172, -5.19552, -2.14030, -1, -0.46722, -0.19247, -0.05277, 0] }
\sshanda{ 8 }{ 2 }{ x^8 + 84x^7 + 1092x^6 + 3630x^5 + 3630x^4 + 1092x^3 + 84x^2 + x }{ [-68.90575, -10.67971, -2.96965, -1,-0.33674, -0.09364, -0.01451, 0]}
\sshanda{ 8 }{ 3 }{x^8 + 154x^7 + 2587x^6 + 9490x^5 + 9490x^4 + 2587x^3 + 154x^2+x}{ [-135.40863, -13.95990, -3.24421, -1, -0.30824, -0.07163, -0.00739,0] }
\sshanda{ 8 }{ 4 }{ x^8 + 210x^7 + 3646x^6 + 13273x^5 + 13273x^4 + 3646x^3 + 210x^2 + x }{ [-191.30191, -14.16374, -3.14005, -1,-0.31847,-0.07060, -0.00523, 0] }
\sshanda{ 8 }{ 5 }{ x^8 + 238x^7 + 4158x^6 + 15115x^5 + 15115x^4+4158x^3+238x^2+x }{[-219.35732, -14.12477, -3.12228, -1,-0.32028, -0.07080, -0.00456, 0] }
\sshanda{ 8 }{ 6 }{ x^8 + 246x^7 + 4278x^6 + 15563x^5 + 15563x^4 + 4278x^3 + 246x^2 + x }{ [-227.49455, -13.97464, -3.13599, -1, -0.31888, -0.07156, -0.00440, 0] }
\sshanda{ 8 }{ 7 }{ x^8 + 247x^7 + 4293x^6 + 15619x^5 + 15619x^4 + 4293x^3 + 247x^2 + x }{ [-228.51096, -13.95665, -3.13765, -1, -0.31870, -0.07165, -0.00438, 0] }
\haha
\ssshanda{ 9 }{ 0 }{ x }{ [0] }
\sshanda{ 9 }{ 1 }{ x^9 + 36x^8 + 336x^7 + 1176x^6 + 1764x^5 + 1176x^4 + 336x^3 + 36x^2 + x }{ [-23.85519, -6.65620, -2.83087, -1.39601, -0.71633, -0.35325, -0.15024, -0.04192, 0] }
\sshanda{ 9 }{ 2 }{ x^9 + 120x^8 + 2310x^7 + 12012x^6 + 20449x^5 + 12012x^4 + 2310x^3 + 120x^2 + x }{ [-97.56308, -15.44345, -4.46160, -1.61264, -0.62010, -0.22414, -0.06475, -0.01025, 0] }
\sshanda{ 9 }{ 3 }{ x^9 + 246x^8 + 6621x^7 + 40116x^6 + 71403x^5 + 40116x^4 + 6621x^3 + 246x^2 + x }{ [-216.23090, -22.10085, -5.15272, -1.67433, -0.59725, -0.19407, -0.04525, -0.00462, 0] }
\sshanda{ 9 }{ 4 }{ x^9 + 372x^8 + 10737x^7 + 64936x^6 + 114962x^5 + 64936x^4 + 10737x^3 + 372x^2 + x }{ [-341.07546, -23.44302, -4.98440, -1.64183, -0.60908, -0.20063, -0.04266, -0.00293, 0] }
\sshanda{ 9 }{ 5 }{ x^9 + 456x^8 + 13402x^7 + 80984x^6 + 143230x^5 + 80984x^4 + 13402x^3 + 456x^2 + x }{ [-424.90556, -23.65930, -4.93995, -1.63740, -0.61072, -0.20243, -0.04227, -0.00235, 0] }
\sshanda{ 9 }{ 6 }{ x^9 + 492x^8 + 14352x^7 + 86678x^6 + 153426x^5 + 86678x^4 + 14352x^3 + 492x^2 + x }{ [-461.29325, -23.25795, -4.94880, -1.64479, -0.60798, -0.20207, -0.04300, -0.00217, 0] }
\sshanda{ 9 }{ 7 }{ x^9 + 501x^8 + 14586x^7 + 88091x^6 + 155946x^5 + 88091x^4 + 14586x^3 + 501x^2 + x }{ [-470.38820, -23.15789, -4.95345, -1.64559, -0.60769, -0.20188, -0.04318, -0.00213, 0] }
\sshanda{ 9 }{ 8 }{ x^9 + 502x^8 + 14608x^7 + 88234x^6 + 156190x^5 + 88234x^4 + 14608x^3 + 502x^2 + x }{ [-471.40751, -23.13604, -4.95662, -1.64474, -0.60800, -0.20175, -0.04322, -0.00212, 0] }
\haha
\ssshanda{ 10 }{ 0 }{ x }{ [0] }
\shanda{ 10 }{ 1 }{ x^{10} + 45x^9 + 540x^8 + 2520x^7 + 5292x^6 + 5292x^5 + 2520x^4 + 540x^3 + 45x^2 + x }{ [-29.30369, -8.28003, -3.60019, -1.83993, -1, -0.54350, -0.27776, -0.12077, -0.03413, 0] }
\shanda{ 10 }{ 2 }{ x^{10} + 165x^9 + 4488x^8 + 34320x^7 + 91091x^6 + 91091x^5 + 34320x^4 + 4488x^3 + 165x^2 + x }{ [-133.20312, -21.41541, -6.35121, -2.40235, -1, -0.41626, -0.15745, -0.04670, -0.00751, 0] }
\shanda{ 10 }{ 3 }{ x^{10} + 375x^9 + 15423x^8 + 145387x^7 + 421769x^6 + 421769x^5 + 145387x^4 + 15423x^3 + 375x^2 + x }{ [-329.52317, -33.53780, -7.79673, -2.59597, -1, -0.38521, -0.12826, -0.02982, -0.00303, 0] }
\shanda{ 10 }{ 4 }{ x^{10} + 627x^9 + 28952x^8 + 275897x^7 + 794694x^6 + 794694x^5 + 275897x^4 + 28952x^3 + 627x^2 + x }{ [-577.70720, -37.60046, -7.62298, -2.51172, -1, -0.39813, -0.13118, -0.02660, -0.00173, 0] }
\shanda{ 10 }{ 5 }{ x^{10} + 837x^9 + 40110x^8 + 382904x^7 + 1101090x^6 + 1101090x^5 + 382904x^4 + 40110x^3 + 837x^2 + x }{ [-786.62666, -38.77026, -7.54646, -2.49648, -1, -0.40056, -0.13251, -0.02579, -0.00127, 0] }
\shanda{ 10 }{ 6 }{ x^{10} + 957x^9 + 45476x^8 + 432834x^7 + 1245792x^6 + 1245792x^5 + 432834x^4 + 45476x^3 + 957x^2 + x }{ [-907.40761, -38.00620, -7.51389, -2.51404, -1, -0.39777, -0.13309, -0.02631, -0.00110, 0] }
\shanda{ 10 }{ 7 }{ x^{10} + 1002x^9 + 47433x^8 + 451199x^7 + 1298925x^6 + 1298925x^5 + 451199x^4 + 47433x^3 + 1002x^2 + x }{ [-952.70805, -37.70097, -7.51573, -2.51739, -1, -0.39724, -0.13305, -0.02652, -0.00105, 0] }
\shanda{ 10 }{ 8 }{ x^{10} + 1012x^9 + 47803x^8 + 454829x^7 + 1309315x^6 + 1309315x^5 + 454829x^4 + 47803x^3 + 1012x^2 + x }{ [-962.84130, -37.55584, -7.52923, -2.51564, -1, -0.39751, -0.13282, -0.02663, -0.00104, 0] }
\shanda{ 10 }{ 9 }{ x^{10} + 1013x^9 + 47840x^8 + 455192x^7 + 1310354x^6 + 1310354x^5 + 455192x^4 + 47840x^3 + 1013x^2 + x }{ [-963.85446, -37.54150, -7.53057, -2.51546, -1, -0.39754, -0.13279, -0.02664, -0.00104, 0] }

\normalsize
\section*{Acknowledgment}
The author would like to thank Einar Steingr\'imsson for sharing his conjecture and also thank the anonymous referee for many helpful suggestions.

\end{document}